\documentclass[12pt,reqno]{amsart}
\usepackage[mathscr]{eucal}
\usepackage{amsfonts,amssymb,amsmath,latexsym,mathabx}
\usepackage{stmaryrd}
\usepackage{accents}
\usepackage[T1]{fontenc}
\usepackage{showlabels}
\usepackage{enumerate}
\usepackage{hyperref}
\hypersetup{
     colorlinks   = true,
     citecolor    = blue,
     linkcolor    = blue
}
\usepackage[left=1in, right=1in, top=1.1in,bottom=1.1in]{geometry}
\usepackage{pdfsync}
\usepackage{verbatim}
\usepackage{enumitem}

\newcommand{\ddi}{\mathrm{d}^\diamond}
\newcommand{\di}{\diamond}
\newcommand{\dd}{\mathrm{d}}

\newcommand{\ist}{\int_{s}^{t}}

\newcommand{\ott}{[0,T]}

\newcommand{\pv}{p-{\rm var}}
\newcommand{\qv}{q-{\rm var}}
\newcommand{\lbar}[1]{\underaccent{\bar}{#1}}

\newcommand{\bd}{\mathbf{D}}
\newcommand{\be}{\mathbf{E}}

\newcommand{\bp}{\mathbf{P}}

\newcommand{\bx}{{\bf x}}

\newcommand{\1}{{\bf 1}}
\newcommand{\2}{{\bf 2}}

\newcommand{\E}{\mathbb E}
\newcommand{\D}{\mathbb D}
\newcommand{\R}{\mathbb R}

\newcommand{\cac}{\mathcal C}

\newcommand{\cf}{\mathcal F}
\newcommand{\ch}{\mathcal H}

\newcommand{\cs}{\mathcal S}

\newcommand{\ep}{\varepsilon}

\newcommand{\ga}{\gamma}

\newcommand{\oom}{\Omega}

\newcommand{\si}{\sigma}

\newcommand{\lp}{\left(}
\newcommand{\rp}{\right)}
\newcommand{\lc}{\left[}
\newcommand{\rc}{\right]}
\newcommand{\lcl}{\left\{}
\newcommand{\rcl}{\right\}}

\newcommand{\lla}{\left\langle}
\newcommand{\rra}{\right\rangle}

\newtheorem{theorem}{Theorem}[section]

\newtheorem{definition}[theorem]{Definition}

\newtheorem{hypothesis}[theorem]{Hypothesis}
\newtheorem{lemma}[theorem]{Lemma}

\newtheorem{proposition}[theorem]{Proposition}

\theoremstyle{remark}
\newtheorem{remark}[theorem]{Remark}

\def\ll{\llbracket}
\def\rr{\rrbracket}

\begin{document}
\title[Skorohod and Stratonovich integrals]
{Skorohod and Stratonovich integrals \\ for controlled processes}
\author{Jian Song and  Samy Tindel}
\address{
	Jian Song: Research Center for Mathematics and Interdisciplinary Sciences,  Shandong University,  Qingdao, Shandong, 266237, China and 
	School of Mathematics, Shandong University, Jinan, Shandong, 250100,  China}
	\email{txjsong@hotmail.com}
	\thanks{J.S. is partially supported by Shandong University grant 11140089963041.}
	
	\address{
	Samy Tindel: Department of Mathematics, Purdue University, 150 N. University Street West Lafayette, Indiana 47907, USA}
\email{stindel@purdue.edu}
\thanks{S.T. is partially supported by the National Science Foundation under
grant DMS-1952966.}

\begin{abstract}
Given a continuous Gaussian process $x$ which gives rise to a $p$-geometric rough path for $p\in (2,3)$, and a general continuous process $y$ controlled by $x$, under proper conditions we establish the relationship between the Skorohod integral $\int_0^t y_s \dd^\diamond x_s$ and the Stratonovich integral $\int_0^t y_s \dd \bx_s$. Our strategy is to employ the tools from rough paths theory and Malliavin calculus to analyze  discrete sums of the integrals. 
\end{abstract}

\maketitle

{
\hypersetup{linkcolor=black}
\tableofcontents 
}

\section{Introduction}
For sake of clarity, we will divide this introduction in 3 parts. In Section~\ref{sec:background} we motivate our problem and recall some previous contributions giving Stratonovich-Skorohod corrections. Section~\ref{sec:main-result} is devoted to a description of our main result, the strategy employed in the article, and some perspectives for future works. At the end, some notations used in this article are introduced in Section~\ref{sec:notation}. 

\begin{comment}
(1)
Motivation for Skorohod-Stratonovich corrections: better links between rough path and stochastic analysis, better analysis of moments for Skorohod.
Recall what has been done \cite{HJT} for $f(B)$, \cite{CL} for solutions to RDEs, Chouk-T EJP in the plane.
Include Section \ref{sec:background} here.
\smallskip
(2)
Description of the results, including a simplified version of the main theorem with the correction formula. This should also include Section \ref{sec:main-result}. 
Comparison with Cass-Lim: broader class of processes with a simpler method.
\smallskip
(3) 
Strategy: based on discrete sums methods, combined with more classical rough paths and Malliavin calculus tools.
Perspective: correction terms for other equations (Volterra, delay, spdes). Rougher cases.
\end{comment}

\subsection{Background}\label{sec:background}

In recent decades, two approaches for the analysis of dynamical systems driven by Gaussian processes have been greatly developed: (i) the ``probabilistic'' approach, which invokes stochastic analysis tools and leads to It\^o-Skorohod integration, and (ii) the ``pathwise'' approach which employs the theory of  rough paths  and gives rise to Stratonovich integration. In general, one gets a more transparent understanding of the system by using the pathwise approach,  while it is more convenient to explore probabilistic properties (e.g. compute the moments for the solution of a noisy  dynamical system driven by a Gaussian noise) via the probabilistic approach.  One key ingredient to understand the connection between these two approaches is the relationship between Skorohod and Stratonovich integrals. 

For a standard Brownian motion, the relationship between  It\^o and Stratonovich integrals is well-known. It is classically obtained by It\^o calculus, although rough paths theory can also be invoked by observing that both It\^o and Stratonovich integrals can be regarded as integrals against rough paths lifted from a Brownian motion with different second order terms. For general Gaussian processes (consider fractional Brownian motion as a typical example), however, it is non-trivial to obtain the relationship. Indeed, for a general Gaussian process It\^o calculus (or martingale calculus) is not available, and moreover Skorohod integrals cannot be regarded as integrals against rough paths lifted from the corresponding Gaussian processes. We briefly recall some  results giving Skorohod-Stratonovich corrections below. 

Let $x=(x^1, \dots, x^d)$ be a $d$-dimensional centered Gaussian process with i.i.d components giving rise to a  $p$-geometric rough path, where $2<p<3$ (see Section~\ref{sec:rough-path-above-X} for more details about geometric rough paths).  Denote by $R$  the covariance function of $x$, namely $R(s,t)=\be[x_t^1x_s^1]$. We also set $R_t=R(t,t).$  The correction terms between Skorohod and Stratonovich integrals with respect to $x$ have been considered in the following cases: 

\begin{enumerate}[wide, labelwidth=!, labelindent=0pt, label=\textit{(\roman*)}]
\setlength\itemsep{.1in}

\item
In \cite{HJT}, the Skorohod-Stratonovich corrections were computed for integrals of the form $\int_{s}^{t}  \nabla f(x_{u}) \ddi x_{u}$ for a smooth function $f$ defined on $\R^{d}$. More specifically, 
\begin{equation}\label{eq:strato-sko-1}
\int_0^t \nabla f(x_r) \dd\bx_r= \int_{0}^{t}  \nabla f(x_{r}) \ddi x_{r} +\frac12\int_0^t \Delta f(x_r) \dd R_r,
\end{equation}
where the integral with respect to $x$ on the left-hand side is a Stratonovich integral while the one on the right-hand side is a Skorohod integral.  The strategy in \cite{HJT} relied on the fact that $\int_{0}^{t}  \nabla f(x_{u}) \ddi x_{u}$ is obtained by taking limits of Riemann-Wick sums of the form:
\begin{equation}\label{eq:riemann-wick-sum}
S^{\Pi_{st},\di}=\sum_{i=0}^{n-1} \sum_{k=1}^{N} \frac{1}{k!}
f^{(k)}(x_{t_i}) \di  {\big(\bx^{\1}_{t_{i} t_{i+1}}\big)}^{\di k},
\end{equation}
where $\di$ stands for the Wick product.
Then the Skorohod-Stratonovich corrections in \cite{HJT} were analyzed thanks to a computation of the Wick corrections in \eqref{eq:riemann-wick-sum}. Notice that an extension of this result to the case of a Gaussian process indexed by $[0,1]^2$ with H\"older exponent greater than $1/3$ was handled in~\cite{SC15}. There some change of variable formulas were derived  for the Stratonovich and Skorohod integrals respectively. As a consequence, the correction terms were computed explicitly. Also note that some preliminary cases for a 1-d fractional Brownian motion had been considered in \cite{NTa}.

\item 
The reference \cite{CL} is concerned with solutions of rough differential equations driven by $x$, where $x$ is again a $d$-dimensional centered Gaussian process with i.i.d components giving rise to a  $p$-geometric rough path (recall that $2<p<3$). The equation can be written as
\begin{equation}\label{eq:rde}
\dd y_r=\sigma(y_r) \dd \bx_r, 
\end{equation}
with a smooth enough coefficient $\si:\R^d\to \R^{d\times d}$, and we refer to \cite{FV-bk} for more details about this object. Below we denote by $y_{0}$ the initial condition of \eqref{eq:rde}, and $J_u^x$ is designated as the Jacobian of the flow map $y_0\to y_u$. Then the formula for the correction terms in~\cite{CL} can be read as
\begin{multline}\label{eq:strato-sko-rde-case}
\int_{0}^{t} y_{r} \dd \bx_{r} 
=
\int_{0}^{t} y_{r} \ddi x_{r} + \frac12\int_0^t\text{tr}[\sigma(y_r)]\dd R_r \\
+\int_{[0< r_1< r_2< t]} \text{tr} \left[J_{r_2}^x(J_{r_1}^x)^{-1} \sigma(y_{r_1})-\sigma(y_{r_2})\right]\dd R(r_1,r_2).
\end{multline}
Consider the $i$-th column $\sigma_i(x)$ of the coefficient matrix $\sigma(x)$ as  a vector field on $\R^d$ for $1\le i\le d$. If the Lie bracket $[\sigma_i, \sigma_j]=\sigma_i\sigma_j-\sigma_j\sigma_i =0$ for $1\le i\le j\le  d$, then the solution $y_t$ to \eqref{eq:rde} is of the form $y_t=\varphi(x_t, y_0)$ with $(\nabla_x\cdot \varphi)(x, y_0)=\text{tr}[\sigma(x, y_0)]$ and $J_t^x=\sigma(y_t)$ (see \cite[Proposition 24]{BC}). Clearly in this case, \eqref{eq:strato-sko-rde-case} coincides with \eqref{eq:strato-sko-1},  noting that the last term on the right-hand side of \eqref{eq:strato-sko-rde-case} now vanishes.  Therefore, relation \eqref{eq:strato-sko-rde-case} is indeed compatible with~\eqref{eq:strato-sko-1}.

\end{enumerate}

\subsection{Main result and strategy}\label{sec:main-result}

In this paper, we consider a  $d$-dimensional centered Gaussian process $x$ with i.i.d components.  Let $y$ be a  \emph{controlled} process relative to $x$. That is, the increments $\delta y_{st}:= y_{t}-y_{s}$ can be decomposed along the increments of $x$ as follows:
\begin{equation}\label{eq:ctrld-proc-intro}
\delta y_{st}^{i}
= \sum_{j=1}^dy_{s}^{x;ij} \, \bx_{st}^{\1;j} +  r_{st}^{i}, ~\text{ for }  i=1, \dots, d,
\end{equation}
where  $y^{x}$ has finite $p$-variation and the remainder $r$ has finite $\frac{p}{2}$-variation (one can alternatively use H\"older spaces in this definition).  Notice that controlled processes are the natural class of functions for which a proper rough integration with respect to $x$ can be constructed (see e.g \cite{Gu}).
The following is the main result of this paper (see Theorem \ref{thm} below for a more precise statement). Under proper conditions on $x$ and $y$, 
\begin{equation}\label{eq:main-result}
\int_{0}^{t} y_r \, \dd \bx_{r}
=
\int_{0}^{t} y_r \ddi x_r 
+\frac12 \sum_{i=1}^d \int_{0}^{t} y_{r}^{x;ii} \dd R_r 
+\sum_{i=1}^d  \int_{[0<r_1<r_2<t]} \left(  \bd^i_{r_{1}} y^i_{r_{2}} -y_{r_{2}}^{x;ii}\right)\dd R(r_{1},r_{2}).
\end{equation}
 On the left-hand side of \eqref{eq:main-result}, the integral $\int_0^t y_r \, \dd \bx_r$ is understood in the rough path sense (see Proposition \ref{prop:intg-ctrld-process} below for further details). On the right-hand side of the same equation, $\int_0^t y_r\ddi x_r$ stands for the Skorohod integral, $y^x$ is defined by \eqref{eq:ctrld-proc-intro}, $R$ is the covariance function alluded to above and $\bd$ represents the Malliavin derivative (notions of Malliavin calculus will be recalled in Section \ref{sec-Mal}). 

Note that our formula \eqref{eq:main-result} unifies the previous cases \eqref{eq:strato-sko-1} and \eqref{eq:strato-sko-rde-case}. Indeed, we have argued that \eqref{eq:strato-sko-rde-case} can be seen as an extension of \eqref{eq:strato-sko-1}. Furthermore, note that the solution $y$ to the rough differential equation \eqref{eq:rde} with a sufficiently regular coefficient function $\sigma(y)$ is a controlled process with $y_s^{x;ij}=(\sigma(y_s))_{ij}$ and  $\bd_sy_t =J_t^x(J_s^x)^{-1} \sigma(y_s)$ for $s\le t$. Therefore it is easy to see that  \eqref{eq:main-result} is an extension of \eqref{eq:strato-sko-rde-case}, which is the main result of \cite{CL}.

Inspired by   \cite{CL,HJT}, our proof of the main result  is  based on the discrete sums method combined with  tools from rough paths theory and Malliavin calculus, 
in which some discrete techniques developed in \cite{LT} are also invoked. We outline the idea as follows. 

 Consider a controlled process $y$ with a decomposition given by \eqref{eq:ctrld-proc-intro},  satisfying some path regularity and Malliavin differentiability conditions.   Let $\pi=\pi^n$ denote the uniform partition of $\ott$ and $\ch$ be the Hilbert space associated to $x$. Denote \begin{equation*}
y^\pi(t)=\sum_{k=0}^{n-1} y_{t_k} \1_{[t_k, t_{k+1}]}(t).
\end{equation*} 
  We first prove that  (see Lemma \ref{lem:con-y})
\begin{equation*}
\lim_{n\to\infty} y^{\pi} = y 
\quad\text{in}\quad
\D^{1,2}(\ch).
\end{equation*}
This enables us to show the convergence of the discrete Skorohod integral $\delta^\diamond(y^\pi)$ to the Skorohod integral $\delta^\diamond(y)=\int_0^T y_{r} \, \ddi x_{r}$ in $L^2(\Omega)$, i.e.,
\begin{equation}\label{eq:sko-sums-intro}
\ist y_{r} \, \ddi x_{r}
=
\lim_{n\to\infty} \sum_{m=0}^{n-1}
\left[ \sum_{i=1}^d y_{t_m}^{i} \di \bx^{\1;i}_{t_{m} t_{m+1}} \right]  \text{ in } L^2(\Omega),
\end{equation}
where we have written $\pi=\{t_0, \ldots, t_n\}$ with $t_m=s+m(t-s)/n$ for $m=0, \ldots, n.$
It is also known from the rough paths theory that the following holds true almost surely:
\begin{equation}\label{eq:strato-sums-intro}
\int_s^t y_r\dd\bx_{r}
=\lim_{|\pi|\to0} \sum_{i=1}^d \sum_{k=0}^{n-1} 
\left(y_{t_k}^i \bx^{\1;i}_{t_k t_{k+1}}+ \frac12\sum_{j=1}^d y_{t_k}^{x; ij} \bx ^{2;ji}_{t_kt_{k+1}}\right),
\end{equation}
where the left-hand side above stands for the rough paths integral of $y$ with respect to $x$. 

The Stratonovich-Skorohod correction terms in \eqref{eq:main-result} now can be obtained by  computing the difference between the right-hand sides in \eqref{eq:sko-sums-intro} and \eqref{eq:strato-sums-intro}. When  computing the difference, one key ingredient will be the forthcoming Proposition \ref{prop:weighted-sum}. This proposition is inspired by the analogous results in \cite{LT} and establishes a general estimate for weighted sums in the second chaos of the Gaussian process $x$.

To end this subsection, we provide some perspectives for future works. On the one hand, as an application, some central limit theorems for Skorohod integrals could be obtained with the help of our main result, generalizing the results in \cite{LT} and \cite{NN}.  On the other hand, noting that in this article the Gaussian rough paths with finite $p$-variation for $p\in(1,3)$ are handled,  we believe that our methodology can be carried out for rougher 
Gaussian paths with $p\ge 3$.  It is also interesting to consider the correction terms for the processes arising from  delay equations (\cite{NNT}), Volterra equations (\cite{DT09, DT11,HT}), etc.

\subsection{Notation}\label{sec:notation}
Let $\pi:0=t_{0}<t_{1}<\cdots<t_{n}=T$ be a partition on $[0,T]$. Take $s,t\in [0,T]$. We write $\ll s,t \rr$ for the discrete interval that consists of   $t_{k} $'s such that   $ t_{k}\in [s,t] $.  We denote by $\cs_{k}([s,t])$ the simplex $\{ (t_{1},\dots, t_{k}) \in [s,t]^{k} ;\, t_{1}\leq \cdots\leq t_{k}  \}$. In contrast, whenever we deal with a discrete interval, we set $\cs_{k}(\ll s,t\rr)=
\{ (t_{1},\dots, t_{k}) \in \ll s,t\rr^{k} ;\, t_{1}< \cdots< t_{k}  \}$. For $t=t_{k}$ we   denote $t-:=t_{k-1}$, $t+:= t_{k+1}$.  We also denote by $\mathcal{D}([s,t])$ the set of all dissections of $[s,t]$.

 For $x=(x^1, \dots, x^n)$ and $y=(y^1,\dots, y^n)$ in $\R^n$, we write write $xy$ for their dot product $x\cdot y=\sum_{i=1}^n x^i y^i$  and write $|x|$ for the Euclid norm $\left(\sum_{i=1}^n x_i^2\right)^{1/2}$.  The $L^p$-norm $(\E[|\xi|^p])^{1/p}$ of a random variable $\xi$ is denoted by $\|\xi\|_p$, for $p\ge1.$
 
 Generally speaking, we will write $C$ for a generic constant whose exact value can change from line to line.

\section{Preliminary material}\label{sec:preliminary-material}

This section contains some basic tools from  rough paths theory and Malliavin calculus, as well as some analytical results, which are crucial for the definition and integration of controlled processes. 

\subsection{Rough path above $\mathbf{x}$}\label{sec:rough-path-above-X}

In this subsection we shall recall the notion of a rough path above a signal $x$, and how this applies to Gaussian signals. The interested reader is referred to \cite{FH,FV-bk,Gu} for further details.

As mentioned in Section \ref{sec:notation}, for $s<t$ and $m\geq 1$, we consider the simplex $\cs_{m}([s,t])=\{(u_{1},\ldots,u_{m})\in\lbrack s,t]^{m};\,u_{1}<\cdots<u_{m}\} $. For notational sake, we just write $\cs_{m}$ for $\cs_{m}(\ott)$. The definition of a rough path above a signal $x$ relies on the following notion of increments.

\begin{definition}\label{def:increments}
Let $k\ge 1$. Then the space of $(k-1)$-increments, denoted by $\cac_k([0,T],\R^d)$ or simply $\cac_k(\R^d)$, is defined as
$$
\cac_k(\R^d)\equiv \lcl g\in C(\cs_{k}; \R^d) ;\, \lim_{t_{i} \to t_{i+1}} g_{t_1 \cdots t_{k}} = 0, \, i\le k-1 \rcl.
$$
\end{definition}

\noindent
We now introduce a finite difference operator called $\delta$, which acts on increments and is useful to split iterated integrals into simpler pieces.
\begin{definition}\label{def:delta-on-C1-C2}
Let $g\in\cac_1(\R^d)$, $h\in\cac_2(\R^d)$. Then for $(s,u,t)\in\cs_{3}$,  we set
\begin{equation*}
  \delta g_{st} = g_t - g_s,
\quad\mbox{ and }\quad
\delta h_{sut} = h_{st}-h_{su}-h_{ut}.
\end{equation*}
\end{definition}

\noindent
The regularity of increments in $\cac_{2}(\R^d)$ will be measured in terms of $p$-variation as follows.
\begin{definition}\label{def:var-norms-on-C2}
For $f \in \cac_2(\R^d)$ and $p>0$, we define
$$
\|f\|_{p-{\rm var}}
=
\|f\|_{p-{\rm var}; [0,T]}
=
\sup_{(t_i)\in\mathcal{D}([0,T])}\left(\sum_{i} |f_{t_{i}t_{i+1}}|^p\right)^{1/p}.
$$
The set of increments in $\cac_{2}(\R^d)$ with finite $p$-variation is denoted by $\cac_{2}^{p{\rm -var}}(\R^d)$.
\end{definition}
\noindent Note that for a continuous  function $g:[0,T]\to \R^d$ with finite  $p$-variation, if we set $\|g\|_{p-{\rm var};[0,T]}=\|\delta g\|_{p-{\rm var};[0,T]}$, then we recover its usual $p$-variation.  

With these preliminary definitions in hand, we can now introduce the notion of a rough path.

\begin{definition}\label{def:RP}
Let $x$ be a continuous $\R^d$-valued  path with finite $p$-variation for some $p\ge 1$. We say that $x$  gives rise to a geometric $p$-rough path if there exists a family
\begin{equation*}
\lcl   
\bx^{n; i_1,\ldots,i_n}_{st} ; \,
(s,t)\in\cs_{2}, \, n\le \lfloor p\rfloor, \, i_1,\ldots,i_n\in\{1,\ldots,d\}
\rcl,
\end{equation*}
such that $\mathbf{x}^{\mathbf{1}}_{st}=\delta x_{st}$ and

\noindent
\emph{(1) Regularity:}
For all $n\le \lfloor p\rfloor$, each component of $\mathbf{x}^{n}$ has finite $\frac{p}{n}$-variation
in the sense of Definition \ref{def:var-norms-on-C2}.

\noindent
\emph{(2) Multiplicativity:}
%%%
With $\delta \mathbf{x}^{n}$ as in Definition \ref{def:delta-on-C1-C2}, we have
\begin{equation}\label{eq:multiplicativity}
\delta \bx^{n; i_1,\ldots,i_n}_{sut}=\sum_{n_1=1}^{n-1}
\bx^{n_{1}; i_1,\ldots,i_{n_1}}_{su}  
\bx^{n-n_{1}; i_{n_1+1},\ldots,i_n}_{ut} .
\end{equation}

\noindent
\emph{(3) Geometricity:}
Let  $x^{\ep}$ be a sequence of piecewise smooth approximations of $x$. For any $n\le \lfloor p\rfloor$ and any set of indices  $i_1,\ldots,i_n\in\{1,\ldots,d\}$, we assume that $\bx^{\ep,n; i_1,\ldots,i_n}$ converges in $\frac{p}{n}$-variation to $\bx^{n; i_1,\ldots,i_n}$, where $\bx^{\ep,n; i_1,\ldots,i_n}_{st}$ is defined for $(s,t)\in\cs_{2}$ by
\begin{equation*}
\bx^{\ep,n; i_1,\ldots,i_n}_{st}
=\int_{(u_{1},\ldots,u_{n})\in\cs_{n}([s,t])} dx_{u_{1}}^{\ep,i_{1}} \cdots dx_{u_{n}}^{\ep,i_{n}}.
\end{equation*}
\end{definition}

We are now ready to state one of the main assumptions on our standing process $x$.

\begin{hypothesis}\label{hyp:nth-iterated-intg-x}
Throughout the paper,  $x$ will designate a  continuous $\R^d$-valued  path with finite $p$-variation for $p\ge 1$. We assume that $x$  gives rise to a geometric rough path in the sense of Definition \ref{def:RP}.
\end{hypothesis}

On top of Hypothesis \ref{hyp:nth-iterated-intg-x},  we assume that $x_t=(x_t^1,\ldots,x_t^d)$ is a continuous centered Gaussian process with i.i.d. components, defined on a complete  probability space $(\Omega, \cf, \bp)$. The covariance function of $x$ is given by
\begin{equation}\label{eq:def-covariance-X}
R(s,t):=\be\lc x_{s}^{j} x_{t}^{j}\rc, 
\end{equation}
for any $j\in \{1, \dots, d\}$. 
%The following notation will be used in the sequel
%\begin{equation}\label{eq:def-variance-Xt}
%\si_{t}^{2} := \be\lc  \lp  x_{t}^{j} \rp^{2} \rc, \quad\text{and}\quad
%\si_{st}^{2} := \be\lc  \lp  \delta x_{st}^{j} \rp^{2} \rc ,
%\end{equation}
%where we recall that the notation of $\delta$ has been introduced in Definition \ref{def:delta-on-C1-C2}. 
Throughout the paper, we will also set $R_t:=R(t,t)$.

The information on the path regularity of $x$ is mostly contained in the rectangular increments $R_{uv}^{st}$ of its covariance function $R$, which are defined as
\begin{equation}\label{eq:rect-increment-cov-fct}
R_{uv}^{st} := \be\lc (x_t^j-x_s^j) \, (x_v^j-x_u^j) \rc.
\end{equation}
The regularity of $R$ is expressed thanks to some $2$d-variation type quantities. For sake of clarity we first  recall the definition of the 2d $\rho$-variation.

\begin{definition}\label{def:2d-variation}
Let $\rho\in[1,\infty)$. For a general continuous function $R:[0,T]^{2}\to\R$, its 2d $\rho$-variation is defined as 
\begin{align} 
\|R\|_{\rho-{\rm var};[s,t]\times[u,v]}:
=\sup_{\substack{(t_{i})\in\mathcal{D}([s,t])\\
(t_{j}^{\prime})\in\mathcal{D}\left(\left[u,v\right]\right)
}
}\left(\sum_{t'_{j}}\sum_{t_{i}}\left|R_{t_{i}t_{i+1}}^{t_{j}^{\prime}t_{j+1}^{\prime}}\right|^{\rho}\right)^{\frac{1}{\rho}}.\label{eq:mixed_var}
\end{align}
where 
\begin{equation}\label{eq:R}
R_{t_{i}t_{i+1}}^{t_{j}^{\prime}t_{j+1}^{\prime}}
=
R(t_{i+1},t_{j+1}^{\prime}) - R(t_{i+1},t_{j}^{\prime}) - R(t_{i},t_{j+1}^{\prime}) + R(t_{i},t_{j}^{\prime}).
\end{equation}

\end{definition}

\noindent
Observe that, whenever the function $R$ in Definition \ref{def:2d-variation} is  a covariance function as in \eqref{eq:def-covariance-X},  the rectangular increment $R_{t_{i}t_{i+1}}^{t_{j}^{\prime}t_{j+1}^{\prime}}$ can also be written as in \eqref{eq:rect-increment-cov-fct}. 

In the following definition, we consider each element $((s,t), (u,v))$ in $\cs_2\times \cs_2$ as a rectangle and denote it by $[s,t]\times[u,v]$. 

\begin{definition}\label{def:2d-control}
A continuous function $\omega: \cs_2\times\cs_2\to\R^+$ is called a 2d control,  if it is zero on degenerate rectangles,  and super-additive in the sense that for all rectangles $A, B$ and $C$ contained in $\cs_2$ satisfying  $ A\cup B\subset C$ and $A\cap B=\emptyset$, 
\[\omega(A)+\omega(B)\le \omega(C).\]
\end{definition}

With these elementary notions at hand, we next introduce a hypothesis which allows the use of both rough paths techniques and tools from stochastic analysis for the underlying process $x$. 

\begin{hypothesis}\label{hyp:2d-var-R}
Let $x$ be a $d$-dimensional continuous and centered Gaussian process with i.i.d.\ components, whose  initial value is 0 and covariance $R$ is given by \eqref{eq:def-covariance-X}.  We assume that for some $\rho\in[1,2)$, the function $R$ admits a finite 2d $\rho$-variation.
\end{hypothesis}

 It is well known that for a continuous  function $g:[0,T]\to \R$ with finite  $p$-variation,  the function $[a,b]\mapsto \|g\|^p_{p-{\rm var};[a,b]}$ is a control. However, for a continuous function $R:[0,T]^2\to \R$ with finite 2d $\rho$-variation,  the function $[a,b]\times [c,d]\mapsto \| R \|^{\rho}_{\rho-{\rm var};[a,b]\times[c,d]}$ may fail to be super-additive for $\rho>1$ (see \cite[Theorem 1]{fv11}). To regain this property, here we introduce the so-called \emph{controlled 2d $\rho$-variation} for $1\le \rho<\infty$ (this notion is also introduced in \cite{fv11}). 
\begin{definition}\label{def:controlled-variation}
Let $\rho\in[1,\infty)$.  For a continuous function $R: [0,T]^2\to \R$, its controlled 2d $\rho$-variation is defined as 
\[\vvvert R \vvvert_{\rho-{\rm var};[s,t]\times [u,v]}:=\sup_{\Pi\in \mathcal P([s,t]\times [u,v])} \left(\sum_{[t_i,t_{i+1}]\times[t'_j,t'_{j+1}] \in \Pi} \left|R_{t_i t_{i+1}}^{t'_j t'_{j+1}}\right|^{\rho}\right)^{\frac1\rho}, \]
where $R_{t_i t_{i+1}}^{t'_j t'_{j+1}}$ is given in \eqref{eq:R}, $\Pi$ is a partition of $[s, t]\times [u,v]$ which is a finite set of essentially disjoint rectangles whose union is $[s,t]\times [u,v]$, and $\mathcal P([s,t]\times [u,v])$ is the collection of all such partitions. 
\end{definition}

The norms $\|\cdot\|$ and $\vvvert\cdot\vvvert$ are  comparable thanks to the following property borrowed from \cite[Theorem 1]{fv11}: for all $\rho'>\rho$ there exists a constant $C_{\rho, \rho'}$ such that
\begin{equation}\label{e:norms-compare}
C_{\rho,\rho'} \vvvert f\vvvert_{\rho'-{\rm var}; [s_1, s_2]\times[t_1, t_2]} \le \|f\|_{\rho-{\rm var}; [s_1, s_2]\times[t_1, t_2]}\le \vvvert f\vvvert_{\rho-{\rm var}; [s_1, s_2]\times[t_1, t_2]}. 
\end{equation} 
 Moreover,  the function $[a,b]\times [c,d] \mapsto \vvvert f \vvvert^\rho_{\rho-{\rm var};[a,b]\times [c,d]}$ is a 2d control (\cite[Theorem 1]{fv11}).

\begin{remark} \label{remark:control-2d-variation}
 Owing to \eqref{e:norms-compare}, any continuous function $R:[0,T]^2\to \R$ with finite \emph{2d $\rho$-variation} also has a finite \emph{controlled 2d  $\rho'$-variation} for all $\rho'>\rho.$ Furthermore,   for all $((s,t), (u,v))\in \cs_2\times \cs_2$,  
\[\|R\|^{\rho'}_{\rho'-{{\rm var}};[s,t]\times[u,v]}\le \omega([s,t]\times[u,v]),\]
where $\omega$ is the 2d control  (as introduced in Definition \ref{def:2d-control}) given by 
\begin{equation}\label{e:w}
\omega([s,t]\times[u,v])= \vvvert R\vvvert^{\rho'}_{{\rho'}-{{\rm var}};[s,t]\times[u,v]}.
\end{equation}
\end{remark}
\begin{remark}
As an example, if the Gaussian process $x$ is a fractional Brownian motion with Hurst parameter $H\in (0, \frac12]$, the covariance $R$ of $x$ has finite 2d $\rho$-variation with $\rho=\frac1{2H}$ and Hypothesis \ref{hyp:2d-var-R} is satisfied (see \cite[Example 1]{fv11}).  If we choose $\rho'>\rho=\frac{1}{2H}$, then the quantity $\|R\|^{\rho'}_{\rho'-{{\rm var}};[s,t]\times[u,v]} $ is controlled by the 2d control $ \vvvert R\vvvert^{\rho'}_{\rho'-{{\rm var}};[s,t]\times[u,v]}$. Note that $ \vvvert R\vvvert_{\rho-{{\rm var}};[0,T]^2}=\infty$ if we choose $\rho\le\frac1{2H}$ (as shown in~\cite[Example 2]{fv11}).
\end{remark}

In the sequel we will also request the function $t\mapsto R(t,t)$ to be H\"older continuous. We now state an additional assumption which guarantees this H\"older continuity (see e.g \cite{CL,GOT} for a similar hypothesis).

\begin{hypothesis}\label{hyp:var-R}
 Let $\rho\in[1,2)$ be given in Hypothesis \ref{hyp:2d-var-R}.   We assume that there exists $C<\infty$ such that for all $s,t\in[0,T]$ the covariance function $R$ satisfies
\begin{equation}\label{eq:R-var}
\|R(t,\cdot)-R(s,\cdot)\|^\rho_{\rho-{\rm{var;}}[0,T]}\le C|t-s|.
\end{equation}
\end{hypothesis}

\begin{remark}\label{remark:holder-R}
A direct consequence of Hypothesis \ref{hyp:var-R} is that $R_t:=R(t,t)$ has   finite $\rho$-variation, by   \cite[Lemma 2.14]{CL}.   Moreover,  recall that by Hypothesis \ref{hyp:2d-var-R}, we have $x_0=0$ and hence $R(0, \cdot)=R(\cdot, 0)\equiv 0$. This together with \eqref{eq:R-var}  implies that $R(t,\cdot)$ and $R(\cdot, t)$ have finite $\rho$-variation for each fixed $t\in[0,T].$
\end{remark}

 \begin{remark}\label{remark:holder-R'}
Given $\rho\in[1,2)$, clearly we have, for $0\le s_1\le s_2\le T$ and $0\le t_1\le t_2\le T$,   
 \[\|R\|_{\rho-\text{var;}[s_1,s_2]\times[t_1,t_2]}^{2} \le \|R\|_{\rho-\text{var;}[s_1,s_2]\times[0,T]}\|R\|_{\rho-\text{var;}[0,T]\times[t_1,t_2]}.\]
Furthermore, it is a direct consequence of \eqref{eq:R-var} that for $0\le s\le t\le T$,
 \begin{equation*}
 \|R\|^\rho_{\rho-\text{var;}[s,t]\times[0,T]}\le C(t-s).
 \end{equation*}
Combining the two inequalities above, we have the following  control on the 2d $\rho$-variation of $R$: for some positive constant $C$, 
\begin{equation}\label{e:bound-R-2rho}
\|R\|_{\rho-\text{var;}[s_1,s_2]\times[t_1, t_2]}^{2\rho} \le C (s_2-s_1)(t_2- t_1).
\end{equation}
\end{remark}

\begin{remark}\label{remark:holder-R''}
Note that for $1\le \gamma\le \gamma'<\infty, \|R\|_{\gamma'-\text{var;}[s,t]\times[u,v]}\le \|R\|_{\gamma-\text{var;}[s,t]\times[u,v]}$. Therefore, under Hypothesis \ref{hyp:var-R},  inequality \eqref{eq:R-var} and hence \eqref{e:bound-R-2rho} hold with $\rho$ replaced by $\rho'\in(\rho, 2)$ and $C$ depending on $(\rho, \rho', T)$.
\end{remark}

 \begin{remark}\label{remark2.10}
Clearly \eqref{e:bound-R-2rho} yields the following relations on squares of the form $[s,t]^2$,
\begin{equation}\label{e:Holder-control-var-R}
\|R\|_{\rho-\text{var;}[s,t]^2}^{\rho} \le C (t-s).
\end{equation}
We say that $R$ has  finite \emph{H\"{o}lder-controlled  2d $\rho$-variation} if $R$ satisfies both Hypothesis \ref{hyp:2d-var-R} and~\eqref{e:Holder-control-var-R}.  An important consequence of $R$ having finite H\"older controlled 2d $\rho$-variation is that $x$ has $1/p$-H\"{o}lder continuous sample paths for every $p>2\rho$. It is also readily checked that, whenever $x$ satisfies \eqref{e:Holder-control-var-R}, we have
\begin{equation}\label{eq:bound-increment-X-L2}
\be\lc \lp \bx_{st}^{\1;i} \rp^{2}  \rc
\le
c\, (t-s)^{\frac{1}{\rho}}.
\end{equation}
\end{remark}

\begin{remark}
Similarly to the argument in \cite[Remark 2.4]{CHLT}, for any process $x$  whose covariance function $R$ admits a finite $\rho$-variation one can introduce a deterministic time-change $\tau:[0,T]\to[0,T]$ such that $\tilde{X}=X\circ\tau$ has finite H\"{o}lder-controlled  2d $\rho$-variation.  That is the time changed process $\tilde X$ satisfies Hypothesis \ref{hyp:2d-var-R}  and equation \eqref{e:Holder-control-var-R}.
\end{remark}

%\begin{remark}\label{rmk: existence rp}
%Since the mixed $(1,\rho)$-variation of $R$ controls $V_\rho(R)$, Proposition \ref{prop:Gaussian-rough-path} and Hypothesis \ref{hyp:2d-var-R} is slightly stronger than the usual assumption that $R$ admits a 2-d $\rho$-variation. However, as pointed out in  \cite{FGGR13}, Hypothesis \ref{hyp:2d-var-R} covers a larger class of examples.
%\end{remark} 

 The following result (stated e.g. in \cite[Theorem 15.33]{FV-bk}) relates the 2d $\rho$-variation of $R$ with the pathwise assumptions allowing to apply the abstract rough paths theory.

\begin{proposition}\label{prop:Gaussian-rough-path}
Let $x=(x^1,\ldots,x^d)$ be a continuous centered Gaussian process with  i.i.d.\ components and covariance function $R$ defined by \eqref{eq:def-covariance-X}. If $R$ satisfies Hypothesis \ref{hyp:2d-var-R}, then $x$ also satisfies Hypothesis \ref{hyp:nth-iterated-intg-x}  provided $p>2\rho$. 
\end{proposition}

Proposition \ref{prop:Gaussian-rough-path} asserts that under Hypothesis \ref{hyp:2d-var-R}, the Gaussian process $x$ is amenable to rough path analysis. In particular, a rough path integral with respect to $x$ can be constructed. In this context, the natural class of integrand one might want to consider is the family of controlled processes. Its definition is recalled below. 

\begin{definition}\label{def:ctrld-process}
Consider a continuous $\R^d$-valued  path $x$ with finite $p$-variation for some $p\ge1$. We say that a continuous $\R^d$-valued  path $y$ of finite $p$-variation is controlled by $x$, if there exist a continuous $\R^{d^2}$-valued path $y^x$ of finite $p$-variation  and  a $1$-increment process $r\in \cac_{2}^{\frac p2{\rm -var}}(\R^d)$ as defined in Definition \ref{def:var-norms-on-C2},  such that 
\begin{equation}\label{eq:ctrld-proc-def}
\delta y_{st}^{i}
= \sum_{j=1}^dy_{s}^{x;ij} \, \bx_{st}^{\1;j} +  r_{st}^{i}, ~\text{ for }  i=1, \dots, d.
\end{equation}
\end{definition}

We are now ready to state the basic integration result for controlled processes, which can be found e.g. in \cite{FH, FV-bk, Gu}. 
\begin{proposition}\label{prop:intg-ctrld-process}
Let $T>0$ be fixed. Let $\bx$ be a geometric $p$-rough path lifted from a continuous $\R^d$-valued path with finite $p$-variation for some $p\in[1, 3),$ and let $y$ be a continuous $\R^d$-valued path of finite $p$-variation that is controlled by $x$ in the sense of Definition \ref{def:ctrld-process}. Then  for $0\le s< t\le T$, one can define the integral $
\int_s^t y_r \dd \bx_r$ as the limit of  the following Riemann sums,
\begin{equation}\label{eq:int-sum}
\int_s^t y_r d\bx_r=\lim_{|\pi_n|\to0} \sum_{k=0}^{n-1} \left(  \sum_{i=1}^d y^i_{t_k}\bx_{t_kt_{k+1}}^{1,i}+\sum_{i=1}^d\sum_{j=1}^d y_{t_k}^{x;ij}\bx_{t_kt_{k+1}}^{2;ij}\right),
\end{equation}
where  $\pi_n=[s=t_0<t_1<\dots<t_n=t]$ is a partition of $[s,t]$ and $|\pi_n|=\max\limits_{k\in\{0,\ldots, n-1\}}|t_{k+1}-t_k|.$ In \eqref{eq:int-sum}, observe that we have also used the convention on inner products put forward in Section \ref{sec:notation}. Moreover, there exists a constant $C=C(T,p)$ depending only on $(T,p)$ such that for all $0\le s<t\le T$ we have
\[\left|\int_s^t y_r d\bx_r -y_s \bx^1_{st}- \sum_{i=1}^d \sum_{j=1}^d y_{s}^{x;ij} \bx_{st}^{2;ij}\right| \le C\left(\|\bx^1\|_{p-{\rm var}}\|r\|_{\frac p2-{\rm var}}+\|\bx^2\|_{\frac p2-{\rm var}}\|y^x\|_{p-{\rm var}} \right)|t-s|^{3/p},
\]
where we recall that $r$ is the increment introduced in \eqref{eq:ctrld-proc-def}.
\end{proposition}

 Recall that our main objective is to compute some Skorohod-Stratonovich corrections as in \cite{CL}. To this aim we will need a more detailed description of the increments of $y$ than the ones given in \eqref{eq:ctrld-proc-def}. Namely we will assume that $y$ is a second order controlled process as defined below (for the definition of  controlled processes of general order, we refer to \cite[Definition 4.17]{FH} or \cite[Definition 5.1]{CHLT}).

\begin{definition}\label{def:ctrld-process'}
Consider a continuous $\R^d$-valued  path $x$ with finite $p$-variation for some $p\ge1$. We say that a continuous $\R^d$-valued  path $y$ of finite $p$-variation is a second-order controlled process with respect to $x$, if there exist a continuous $\R^{d^2}$-valued path $y^x$, a continuous $\R^{d^3}$-valued path $y^{xx}$, both of which are of finite $p$-variation,  and $1$-increment processes $r\in \cac_{2}^{\frac p3{\rm -var}}(\R^d), r^x \in \cac_{2}^{\frac p2{\rm -var}}(\R^{d^2})$ as defined in Definition \ref{def:var-norms-on-C2},  such that  for  $i=1,\dots, d$ and $(s,t)\in \cs_2([0,T])$ we have
\begin{equation}\label{e:y-1}
\delta y_{st}^{i}
= \sum_{j=1}^dy_{s}^{x;ij} \, \bx_{st}^{\1;j} +  \sum_{j,k=1}^dy_{s}^{xx;ijk} \, \bx_{st}^{\2;jk}+ r_{st}^{i}. 
\end{equation}
 In addition, the increment $y^x$ in \eqref{e:y-1} is a controlled process of order 1, that is for $i, j=1, \dots, d$ and $(s,t)\in \cs_2([0,T])$ we have
\begin{equation}\label{e:y-2}
\delta y_{st}^{x;ij}
= \sum_{k=1}^dy_{s}^{xx;ijk} \, \bx_{st}^{\1;k} +  r_{st}^{x;ij}.
\end{equation}
\end{definition}

\subsection{Higher dimensional Young integrals}
In this subsection, we gather some inequalities for Young integrals in $\R^n$ which will feature in our computations throughout the paper. We start by a relation for integrals in the plane borrowed form \cite{fv10, Tow02}. 

\begin{theorem}\label{2d-Young}
Let $f, R:[0,T]^2\to \R$ be continuous functions  with finite $p$-variation and finite $q$-variation respectively for $\frac1p+\frac1q>1$. Specifically recalling our Definition \ref{def:2d-variation}, we assume $\|f\|_{p-{\rm var};[0,T]^2}<\infty$ and $\|R\|_{q-{\rm var};[0,T]^2}<\infty$. Moreover,  assume that for all $s_1,s_2\in[0,T]$, both $f(s_1,\cdot)$ and $f(\cdot, s_2)$ have finite 1-dimensional $p$-variation as given in Definition~\ref{def:var-norms-on-C2}. Then the $2$d Young-Stieltjes integral of $f$ with respect to $R$  exists and the following Young's inequality holds, for $[\lbar{s}_1,\bar{s}_1]\times[\lbar{s}_2,\bar{s}_2]\subset[0,T]^2$,
\begin{multline}\label{eq:bound-2d}
\left|\int_{[\lbar{s}_1,\bar{s}_1]\times[\lbar{s}_2,\bar{s}_2]} f(s_1,s_2) dR(s_1,s_2)\right|
\le C_{p,q}\Bigg(|f(\lbar{s}_1,\lbar{s}_2)|+\|f(\lbar{s}_1,\cdot)\|_{\pv;[\lbar{s}_2,\bar{s}_2]}
\\
+\|f(\cdot, \lbar{s}_2)\|_{\pv; [\lbar{s}_1,\bar{s}_1]}
+\|f\|_{\pv;[\lbar{s}_1,\bar{s}_1]\times[\lbar{s}_2,\bar{s}_2]}\Bigg)  \| R\|_{\qv;[\lbar{s}_1,\bar{s}_1]\times[\lbar{s}_2,\bar{s}_2]}\, .
\end{multline}
\end{theorem}

We now state a lemma about integration in $\R^4$ which will be invoked in order to analyze discretization properties for the Malliavin derivative of a controlled process $y$. Although its proof might be traced back to \cite{Tow02}, we include it here for the sake of clarity since Lemma~\ref{4d-Young} is tailored for our specific needs.

\begin{lemma}\label{4d-Young} Let $f,g,R$ be continuous functions defined on $[0,T]^2$. Similarly to Theorem~\ref{2d-Young}, we assume that $f,g$ have finite $p$-variation, as well as $f(s_1,\cdot), f(\cdot, s_3), g(s_2, \cdot)$ and $g(\cdot, s_4)$ for fixed arbitrary $s_1,s_2,s_3, s_4\in[0,T]$. We also suppose that $R$ has finite $q$-variation on $[0,T]^2$, with $p,q$ satisfying $\frac1p+\frac1q>1$. Then for $\lbar{s}_1, \bar{s}_1, \ldots, \lbar{s}_4, \bar{s}_4\in[0,T]$ such that $\lbar{s}_j< \bar{s}_j$ for $j=1,\dots,4$, the following Young integral in $\R^4$ is well defined: 
\[
I^{f,g,R}(\lbar{s}_1, \bar{s}_1, \ldots, \lbar{s}_4, \bar{s}_4):=\int_{[\lbar{s}_1, \bar{s}_1]\times[\lbar{s}_2, \bar{s}_2]\times[\lbar{s}_3, \bar{s}_3]\times[\lbar{s}_4, \bar{s}_4]} f(s_1,s_3) g(s_2,s_4) dR(s_1,s_2) \dd R(s_3,s_4).
\]
Moreover, $I^{f,g,R}(\lbar{s}_1, \bar{s}_1, \ldots, \lbar{s}_4, \bar{s}_4)$ can be bounded as 
\begin{align}
&\left|I^{f,g,R}(\lbar{s}_1, \bar{s}_1, \ldots, \lbar{s}_4, \bar{s}_4)\right|
\le  C_{p,q}~ \|R\|_{\qv;[\lbar{s}_1,\bar{s}_1]\times[\lbar{s}_2,\bar{s}_2]} \|R\|_{\qv;[\lbar{s}_3,\bar{s}_3]\times[\lbar{s}_4,\bar{s}_4]} \notag\\
& \quad \times\Big(|f(\lbar{s}_1,\lbar{s}_3)|+\|f(\cdot,\lbar{s}_3)\|_{\pv;[\lbar{s}_1,\bar{s}_1]}+\|f(\lbar{s}_1,\cdot)\|_{\pv;[\lbar{s}_3,\bar{s}_3]}+\|f\|_{\pv;[\lbar{s}_1,\bar{s}_1]\times[\lbar{s}_3,\bar{s}_3]}\Big)\notag\\
& \quad \times\Big(|g(\lbar{s}_2,\lbar{s}_4)|+\|g(\cdot,\lbar{s}_4)\|_{\pv;[\lbar{s}_2,\bar{s}_2]}+\|g(\lbar{s}_2,\cdot)\|_{\pv;[\lbar{s}_4,\bar{s}_4]}+\|g\|_{\pv;[\lbar{s}_2,\bar{s}_2]\times[\lbar{s}_4,\bar{s}_4]}\Big)\,. \label{eq:bound-4d}
\end{align}
\end{lemma}

\begin{proof} We will divide this proof in several steps. 

\noindent{\bf Step 1:} Decomposition of the integral.  We can write 
\begin{equation}\label{eq:decom-I}
I^{f,g,R}(\lbar{s}_1, \bar{s}_1, \ldots, \lbar{s}_4, \bar{s}_4)=\int_{[\lbar{s}_3, \bar{s}_3]\times[\lbar{s}_4,\bar{s}_4]}F(s_3, s_4) \dd R(s_3, s_4)
\end{equation}
where the function $F$ is defined on $[0,T]^2$ by 
\begin{equation}\label{eq:F}
F(s_3,s_4)=\int_{[\lbar{s}_1, \bar{s}_1]\times[\lbar{s}_2, \bar{s}_2]} f(s_1,s_3) g(s_2,s_4)\dd R(s_1,s_2),
\end{equation}
and where we observe that  the right-hand side of \eqref{eq:F} is well  defined thanks to Theorem~\ref{2d-Young}.    Our strategy in order to estimate $I^{f,g,R}$ will rely on some succesive applications of \eqref{eq:bound-2d}. 
Specifically, with \eqref{eq:decom-I} in mind, relation \eqref{eq:bound-2d} yields
\begin{multline}\label{eq:bound-I}
\left|I^{f,g,R}(\lbar{s}_1, \bar{s}_1, \lbar{s}_2, \bar{s}_2, \lbar{s}_3, \bar{s}_3, \lbar{s}_4, \bar{s}_4)\right|
\le C_{p,q} \Big(|F(\lbar{s}_3, \lbar{s}_4)|
+\|F(\lbar{s}_3,\cdot)\|_{\pv,[\lbar{s}_4, \bar{s}_4]} \\ 
+\|F(\cdot, \lbar{s}_4)\|_{\pv,[\lbar{s}_3,\bar{s}_3]} 
 +\|F\|_{\pv,[\lbar{s}_3, \bar{s}_3]\times[\lbar{s}_4,\bar{s}_4]}\Big) \| R\|_{\qv,[\lbar{s}_3, \bar{s}_3]\times[\lbar{s}_4,\bar{s}_4]}. 
\end{multline}
We will now estimate the terms in right-hand side of \eqref{eq:bound-I} separately. 

\noindent{\bf Step 2:} Upper bound for $F(\lbar{s}_3, \lbar{s}_4)$. Given $(\lbar{s}_3, \lbar{s}_4)\in[0,T]^2$ and recalling the definition \eqref{eq:F} of $F$, another application of \eqref{eq:bound-2d} enables to write
 \begin{multline*}
|F(\lbar{s}_3,\lbar{s}_4)|
\le 
C_{p,q}\Big( |f(\lbar{s}_1, \lbar{s}_3)g(\lbar{s}_2, \lbar{s}_4)|
+ |f(\lbar{s}_1,\lbar{s}_3)|\|g(\cdot, \lbar{s}_4)\|_{\pv,[\lbar{s}_2, \bar{s}_2]} \\
+|g(\lbar{s}_2, \lbar{s}_4)|\|f(\cdot, \lbar{s}_3)\|_{\pv,[\lbar{s}_1, \bar{s}_1]} 
+\|f(\cdot, \lbar{s}_3)\|_{\pv,[\lbar{s}_1,\bar{s}_1]}\|g(\cdot, \lbar{s}_4)\|_{\pv,[\lbar{s}_2,\bar{s}_2]} \Big)
\|R\|_{\qv,[\lbar{s}_1,\bar{s}_1]\times[\lbar{s}_2,\bar{s}_2]}
\end{multline*}
and we notice that the above expression can be simplified as 
 \begin{align}
|F(\lbar{s}_3,\lbar{s}_4)|\le & C_{p,q}\Big( |f(\lbar{s}_1, \lbar{s}_3)| +\|f(\cdot, \lbar{s}_3)\|_{\pv,[\lbar{s}_1, \bar{s}_1]}\Big)\notag\\
&\qquad \times \Big(|g(\lbar{s}_2, \lbar{s}_4)|+\|g(\cdot, \lbar{s}_4)\|_{\pv,[\lbar{s}_2, \bar{s}_2]} \Big)\|R\|_{\qv,[\lbar{s}_1,\bar{s}_1]\times[\lbar{s}_2,\bar{s}_2]}. \label{eq:bound-F1}
\end{align}

\noindent{\bf Step 3:} Upper bound for  $\|F(\lbar{s}_3,\cdot)\|_{\pv,[\lbar{s}_4,\bar{s}_4]}$. Recall the Definition \ref{def:var-norms-on-C2} of $p$-variation. We thus have
\[\|F(\lbar{s}_3,\cdot)\|_{\pv,[\lbar{s}_4, \bar{s}_4]}=\sup_{\pi} \left(\sum_i |F(\lbar{s}_3, v_{i+1})-F(\lbar{s}_3, v_i)|^p\right)^{1/p}.
\]
Plugging expression \eqref{eq:F} into the above relation, we get 
\[ \|F(\lbar{s}_3,\cdot)\|_{\pv,[\lbar{s}_4, \bar{s}_4]}^{p}
=\sup_{\pi}\sum_{i}\left| \int_{[\lbar{s}_1,\bar{s}_1]\times[\lbar{s}_2,\bar{s}_2]} f (s_1, \lbar{s}_3) \Big(g(s_2,v_{i+1})-g(s_2, v_i)  \Big) dR(s_1,s_2)\right|^p .
\]
We now apply \eqref{eq:bound-2d} again and we end up with
\begin{equation}\label{eq:bound-F2}
\|F(\lbar{s}_3,\cdot)\|_{\pv,[\lbar{s}_4, \bar{s}_4]}\le C_{p,q}\|R\|_{\qv,[\lbar{s}_1, \bar{s}_1]\times[\lbar{s}_2, \bar{s}_2]}\sum_{k=1}^4V_k ,
\end{equation}
where the terms $V_1,V_{2}$ are respectively defined by 
\begin{align*}
V_1&= |f(\lbar{s}_1,\lbar{s}_3)|\sup_\pi \left(\sum_{i}\left|g(\lbar{s}_2,v_{i+1})-g(\lbar{s}_2,v_i)\right|^p\right)^{1/p};\\
V_2 &= |f(\lbar{s}_1, \lbar{s}_3)| \sup_{\pi} \left(\sum_{i} \|g(\cdot, v_{i+1})-g(\cdot, v_i)\|_{\pv,[\lbar{s}_2, \bar{s}_2]}^p\right)^{1/p},
\end{align*}
and similarly the terms $V_{3},V_{4}$ are expressed as
\begin{align*}
V_3&=\sup_{\pi} \left(\sum_{i} |g(\lbar{s}_2, v_{i+1})-g(\lbar{s}_2, v_i)|^p \|f(\cdot, \lbar{s}_3) \|_{\pv,[\lbar{s}_1, \bar{s}_1]}^p\right)^{1/p};\\
V_4& =\sup_{\pi} \left(\sum_{i} \| f(\cdot, \lbar{s}_3) (g(*, v_{i+1})-g(*, v_i))\|_{\pv,[\lbar{s}_1,\bar{s}_1]\times[\lbar{s}_2,\bar{s}_2]}^p\right)^{1/p}.
\end{align*}
In addition, the terms $V_1,V_2, V_3$ are easily bounded. Indeed, resorting again to Definition~\ref{def:var-norms-on-C2}, we get
\begin{equation}\label{eq:V1}
V_1=|f(\lbar{s}_1, \lbar{s}_3)| \|g(\lbar{s}_2, \cdot)\|_{\pv,[\lbar{s}_4, \bar{s}_4]},
\qquad
V_2\le|f(\lbar{s}_1, \lbar{s}_3)| \|g\|_{\pv,[\lbar{s}_2, \bar{s}_2]\times[\lbar{s}_4,\bar{s}_4]},
\end{equation}
and
\begin{equation}\label{eq:V3} 
V_3=\|f(\cdot, \lbar{s}_3)\|_{\pv,[\lbar{s}_1, \bar{s}_1]}\|g(\lbar{s}_2, \cdot)\|_{\pv,[\lbar{s}_4, \bar{s}_4]} \, .
\end{equation}
For the term $V_4$, by Definition \ref{def:2d-variation}, it is  readily checked that 
\begin{align}
V_4&\le \|f(\cdot,\lbar{s}_3)\|_{\pv,[\lbar{s}_1,\bar{s}_1]} \|g\|_{\pv,[\lbar{s}_2,\bar{s}_2]\times[\lbar{s}_4,\bar{s}_4]}\,. \label{eq:V4}
\end{align}
Hence, plugging \eqref{eq:V1}, \eqref{eq:V3} and \eqref{eq:V4} into \eqref{eq:bound-F2}, we end up with
 \begin{multline}\label{eq:bound-F3}
\|F(\lbar{s}_3, \cdot)\|_{\pv,[\lbar{s}_4,\bar{s}_4]}
\le C_{p,q} \Big(|f(\lbar{s}_1, \lbar{s}_3)|+\|f(\cdot, \lbar{s}_3)\|_{\pv,[\lbar{s}_1, \bar{s}_1]}|\Big)
\\
\Big(\|g(\lbar{s}_2, \cdot)\|_{\pv,[\lbar{s}_4,\bar{s}_4]} 
+ \|g\|_{\pv,[\lbar{s}_2, \bar{s}_2]\times[\lbar{s}_4,\bar{s}_4]}\Big)
\|R\|_{\qv,[\lbar{s}_1, \bar{s}_1]\times[\lbar{s}_2, \bar{s}_2]} .
\end{multline}
Furthermore, notice that in a similar way we get
 \begin{multline}\label{eq:bound-F4}
\|F(\cdot,\lbar{s}_4)\|_{\pv,[\lbar{s}_4,\bar{s}_4]} 
\le C_{p,q} \Big(|g(\lbar{s}_2, \lbar{s}_4)|
+\|g(\cdot, \lbar{s}_4)\|_{\pv,[\lbar{s}_2, \bar{s}_2]}|\Big)\\
\Big(\|f(\lbar{s}_1, \cdot)\|_{\pv,[\lbar{s}_3,\bar{s}_3]} 
+ \|f\|_{\pv,[\lbar{s}_1, \bar{s}_1]\times[\lbar{s}_3,\bar{s}_3]}\Big)\|R\|_{\qv,[\lbar{s}_1, \bar{s}_1]\times[\lbar{s}_2, \bar{s}_2]} .\end{multline}

\noindent{\bf Step 4:} Upper bound for  $\|F\|_{\pv,[\lbar{s}_3,\bar{s}_3]\times[\lbar{s}_4,\bar{s}_4]}$.  According to Definition \ref{def:2d-variation}, one can write 
\[\|F\|^p_{\pv,[\lbar{s}_3,\bar{s}_3]\times[\lbar{s}_4,\bar{s}_4]} =\sup_{\pi}\sum_{t_i, t_j'} \left|F(t_i, t'_j)+F(t_{i+1}, t'_{j+1})-F(t_i, t'_{j+1})-F(t_{i+1}, t_j')\right|^p\, ,\]
where we recall that $\pi$ takes the form $\pi\in\mathcal D([\lbar{s}_3, \bar{s}_3])\times \mathcal D([\lbar{s}_4, \bar{s}_4])$
and the notation $\mathcal D([s,t])$ is introduced in Section \ref{sec:notation}.  Hence with the expression \eqref{eq:F} of $F$ in mind we get
\begin{multline*}
\|F\|^p_{\pv,[\lbar{s}_3,\bar{s}_3]\times[\lbar{s}_4,\bar{s}_4]} \\
=\sup_{\pi}\sum_{t_i, t_j'} 
\left|\int_{[\lbar{s}_1,\bar{s}_1]\times
[\lbar{s}_2,\bar{s}_2]} (f(s_1, t_{i+1})-f(s_1, t_i))(g(s_2, t_{j+1}'-g(s_2, t_j'))dR(s_1,s_2) \right|^p.
\end{multline*} 
In this context, relation \eqref{eq:bound-2d} can thus be read as 
\begin{equation*}
\|F\|_{\pv,[\lbar{s}_3,\bar{s}_3]\times[\lbar{s}_4,\bar{s}_4]}
\le C \|R\|_{\qv,[\lbar{s}_1,\bar{s}_1]\times[\lbar{s}_2,\bar{s}_2]}\sup_{\pi}\Bigg(\sum_{t_i, t_j'}
\left|Q_{ij'}\right|^p\Bigg)^{1/p},
\end{equation*}
 where the term $Q_{ij'}$ is defined by 
\begin{align*}
&Q_{ij'}
= |(f(\lbar{s}_1, t_{i+1})-f(\lbar{s}_1, t_i))(g(\lbar{s}_2, t_{j+1}')-g(\lbar{s}_2, t_j'))|\\
&\quad +|f(\lbar{s}_1, t_{i+1})-f(\lbar{s}_1, t_i)| \|g(\cdot, t'_{j+1})-g(\cdot, t'_j)\|_{\pv,[\lbar{s}_2,\bar{s}_2]}\\
&\quad+\|f(\cdot, t_{i+1}) -f(\cdot, t_i)\|_{\pv,[\lbar{s}_1,\bar{s}_1]} |g(\lbar{s}_2, t'_{j+1})-g(\lbar{s}_2, t'_j)|\\
&\quad +\|f(\cdot, t_{i+1}) -f(\cdot, t_i)\|_{\pv,[\lbar{s}_1,\bar{s}_1]}\|g(\cdot, t'_{j+1})-g(\cdot, t'_j)\|_{\pv,[\lbar{s}_2,\bar{s}_2]},
\end{align*}
and we notice that $Q_{ij'}$ can easily be simplified as 
\begin{align*}
Q_{ij'}
=&\Big(|(f(\lbar{s}_1, t_{i+1})-f(\lbar{s}_1, t_i))|+\|f(\cdot, t_{i+1}) -f(\cdot, t_i)\|_{\pv,[\lbar{s}_1,\bar{s}_1]}\Big)\\
&~~~\times \Big(|(g(\lbar{s}_2, t_{j+1}')-g(\lbar{s}_2, t_j'))|+\|g(\cdot, t'_{j+1})-g(\cdot, t'_j)\|_{\pv,[\lbar{s}_2,\bar{s}_2]}\Big)\, .
\end{align*}
Summarizing our computations in this step, we have found that 
\begin{multline}
\|F\|_{\pv,[\lbar{s}_3,\bar{s}_3]\times[\lbar{s}_4,\bar{s}_4]} 
\le  C\|R\|_{\qv,[\lbar{s}_1,\bar{s}_1]\times[\lbar{s}_2,\bar{s}_2]} 
\Big(\|f(\lbar{s}_1, \cdot)\|_{\pv, [\lbar{s}_3,\bar{s}_3]}+\|f\|_{\pv,[\lbar{s}_1,\bar{s}_1]\times[\lbar{s}_3,\bar{s}_3]}\Big) \\
\times
\Big(g(\lbar{s}_2, \cdot)\|_{\pv, [\lbar{s}_4,\bar{s}_4]}+\|g\|_{\pv,[\lbar{s}_2,\bar{s}_2]\times[\lbar{s}_4,\bar{s}_4]}\Big)\,.\label{eq:bound-F5}
\end{multline}

\noindent{\bf Step 5:} Conclusion. Let us gather our estimates  \eqref{eq:bound-F1}, \eqref{eq:bound-F3}, \eqref{eq:bound-F4} and \eqref{eq:bound-F5}  into \eqref{eq:bound-I}. Then we let the patient reader check that \eqref{eq:bound-4d} is achieved. This finishes the proof. 
\end{proof}

\subsection{The Hilbert space associated to $x$}\label{sec:wiener-space-general}

Consider a continuous d-dimensional centered Gaussian process $x$ on $[0, T]$ with covariance function $R$ given by \eqref{eq:def-covariance-X}. Every component of $x$ (say $x^1$) is a 1-dimensional centered Gaussian process with covariance $R$. In this section we review some basic facts about the related Hilbert space $\ch$ of functions for which Wiener integrals with respect to $x$ (see e.g. \cite{Nu06}) are well defined.

The Hilbert space $\ch$ is the completion of  the set of step functions
\[
\mathcal{E}
=
\left\{  \sum_{i=1}^{n}a_{i} \1_{\left[  0,t_{i}\right]  }:a_{i}\in%
\mathbb{R}
\text{, }t_{i}\in\left[  0,T\right], i=1,\dots, n \text{ for } n\in\mathbb N    \right\}  ,
\]
with respect to  the inner product
\begin{equation*}
\left\langle \sum_{i=1}^{n} a_{i} \1_{[0,t_{i}]}  ,
\sum_{j=1}^{m}b_{j} \1_{[0,s_{j}]}  \right\rangle _{\mathcal{H}}
=
\sum_{i=1}^{n}\sum_{j=1}^{m}a_{i}b_{j}R\left(  t_{i},s_{j}\right).
\end{equation*}
Observe that this inner product can also be written as
\begin{align}\label{eq:def-inner-pdt-H}
\left\langle \sum_{i=1}^{n} a_{i} \1_{[0,t_{i}]}  ,
\sum_{j=1}^{m}b_{j} \1_{[0,s_{j}]}  \right\rangle _{\mathcal{H}}
= &\int_0^T\int_0^T\left( \sum_{i=1}^{n}a_{i}\1_{[0,t_i]}(t)\right)\left( \sum_{j=1}^{m} b_j\1_{[0,s_j]}(s)\right) ~dR(t,s).
\end{align}
One can further relate $\ch$ to our driving process $x$ in the following way:
let $\mathbf H $ be the closure of the set 
\[
{\bf \mathsf{E}}=\left\{  \sum\nolimits_{i=1}^{n}a_{i}x^1_{t_{i}}:a_{i}\in
\mathbb{R}
,\text{ }t_{i}\in\left[  0,T\right]  , i=1, \dots, n \text{ for  }n\in%
\mathbb{N}
\right\} ,
\]
in $L^2(\Omega, \cf, \mathbf P)$. Then the linear map $x^1: \mathcal E\to \mathsf E$ defined by $x^1(\1_{[0,t]})=x^1_t$ extends to a linear isometry between $\mathcal H$ and $\mathbf H$. Hence,  $\mathbf H=\{x^1(h), h\in \mathcal H\}$ and this family is known as the isonormal Gaussian process related to $x^1$ (see \cite[Definition 1.1.1]{Nu06}). Note that $x^1(h)$  for $h\in \mathcal H$ is called the Wiener integral of $h$ with respect to $x^1$ and is usually denoted by $\int_0^T h(s)dx^1_s$.

{
\begin{remark}\label{representation H norm}
Recall that we have assumed $x_0=0$ and thus $R(0,0)=0$. Thus relation~\eqref{eq:def-inner-pdt-H} suggests 
\begin{align}\label{rep H norm}
\langle h_1, h_2\rangle_\mathcal{H}=\int_0^T\int_0^Th_1(s)h_2(t)dR(s,t) ~~\text{ for } h_1, h_2\in\mathcal{H},
 \end{align}
whenever the 2D Young's integral on the right-hand side is well-defined (see, e.g., \cite[Proposition 4]{CFV} for details).
\end{remark}
}

\begin{remark}\label{rmk:H-on-subinterval}
Denoting by $\mathcal{E}([a,b])$ the set of step functions in $\mathcal E$  restricted on $[a,b]\subset[0,T]$,  the closure $ \mathcal H([a,b])$ of  $\mathcal E([a,b])$ with respect to the inner product \eqref{eq:def-inner-pdt-H}  then coincides with $\ch$ restricted on $[a,b]$, and for $f,g\in \ch$, 
\begin{equation}\label{eq:norm-H-as-2d-young}
\lla f \, \1_{[a,b]} , \, g \, \1_{[a,b]} \rra_{\ch}
=
\lla f  , \, g  \rra_{\ch([a,b])}.
\end{equation}
\end{remark}

\subsection{Malliavin calculus for Gaussian processes}\label{sec-Mal}

In this subsection, we collect some basic concepts of Malliavin calculus, and we refer  to \cite{Nu06} for more details.

Recall that  $x_t$ is a continuous centered $d$-dimensional Gaussian process with i.i.d.\ components, defined on a complete probability space $(\Omega, \cf, \bp)$. For the sake of simplicity, we assume  that $\cf$ coincides with the $\sigma$-algebra  generated by $\{x_{t}; \, t\in[0,T]\}$. For the $d$-dimensional process $x$, we define an extension of the Wiener integral defined as follows: let $\varphi=(\varphi^1,\ldots, \varphi^d)$ be an element of $\ch^d$ where we recall that $\ch$ has been introduced in Section~\ref{sec:wiener-space-general}.  Then we set 
\begin{equation}\label{e:Wiener-int}
x(\varphi)=\sum_{j=1}^dx^j(\varphi^j)\,, 
\end{equation}
where each term $x^j(\varphi^j)$ is a 1-d Wiener integral as in Section \ref{sec:wiener-space-general}.

A smooth functional of $x$ is a random variable of the form $F=f(x(\varphi_1), \ldots, x(\varphi_n))$, where $n\ge 1$, $\{\varphi_1,\ldots, \varphi_n\}$ is a family of elements of $\ch^d$ and each $x(\varphi_i)$ is understood as in \eqref{e:Wiener-int}. Moreover, we assume that the function $f:\R^n\to\R$ is smooth and its partial derivatives grow at most polynomially fast. Then, the Malliavin derivative $\mathbf DF$ of $F$ is the $\ch^d$-valued random variable defined by 
\begin{equation}\label{e:DF}
\mathbf DF=\sum_{k=1}^n\frac{\partial f}{\partial x_k}(x(\varphi_1),\dots, x(\varphi_n)) \varphi_k. 
\end{equation}
One can show that  $\mathbf D$ is closable from $L^2(\Omega)$ to $L^2(\Omega; \ch)$, and thus one may span the space of the smooth and cylindrical random variables under the norm
\[\|F\|_{1,2}=\left(\E[F^2]+\E[\|\mathbf DF\|_\ch^2]\right)^{\frac12}\,. \]
The resulting closure is called Sobolev space $\mathbb D^{1,2}$. 

\begin{remark}\label{remark:DF}
As seen in \eqref{e:DF}, the Malliavin derivative $\mathbf DF$ of a functional $F$ is a $\R^d$-valued process. The $i$-th coordinate of $\mathbf DF$ corresponds to the Malliavin derivative of $F$ with respect to the randomness  in $x^i$ only. It will be denoted by $\mathbf D^i F$ in the sequel. 
\end{remark}

The divergence operator $\delta^\diamond$ (also known as the Skorohod integral) is the adjoint operator of the Malliavin derivative operator $\mathbf D$ defined by the duality  relation
\[\E[F\delta^\diamond(u)]=\E[\langle \mathbf DF, u\rangle_{\ch^d}],~~\text{ for all }  F\in \mathbb D^{1,2} \text{  and } \text{ for all }  u\in \text{Dom } \delta^\diamond.\]
 Here $\text{Dom } \delta^\diamond$ is the domain of the divergence operator $\delta^\diamond$, which is the space of $\ch$-valued random variables $u\in L^2(\Omega; \ch^d)$ such that  $|\E[\langle \mathbf DF, u\rangle_{\ch^d}]|\le c_F \|F\|_2$ with some constant $c_F$ depending on $F$, for all $F\in \mathbb D^{1,2}$. In particular, $\mathbb D^{1,2}(\ch^d)\subset \text{Dom } \delta^\diamond.$ Note that  for $u\in \text{Dom } \delta^\diamond$, we have $\delta^\diamond(u)\in L^2(\Omega)$ and $\E [\delta^\diamond(u)]=0$.  By convention,  we also take the following notation, for $u\in \text{Dom } \delta^\diamond$,
\begin{equation}\label{e:divergence}
\int_0^Tu_t \, \ddi x_t:=\delta^\diamond(u).
\end{equation}

For our main computations below we shall invoke the following relation taken from \cite{NTa}: for any $G\in\mathbb D^{1,2}(\R^d)$ and $0\le a <b \le T$ we have 
\begin{equation}\label{e:relation}
\delta^\diamond(G \, \1_{[a,b]})=\sum_{i=1}^d \int_a^b G^i \, \dd^\diamond x_t^i=\sum_{i}^d G^i\diamond \delta x_{ab}^i,
\end{equation}
where $\diamond$ stands for the Wick product (see \cite{HJT} for a brief account on Wick products). Moreover, according to \cite[Proposition 4.7]{HY}, relation \eqref{e:relation} can be simplified as 
\begin{equation}\label{e:relation'}
\delta^\diamond(G \, \1_{[a,b]}) =\sum_{i=1}^d G^i \, \delta x_{ab}^i-\langle \mathbf D^i G^i, \1_{[a,b]}\rangle_\mathcal H. 
\end{equation}

\subsection{Discrete rough paths techniques}

In this subsection, we develop some inequalities about discrete sums in a rough paths context. This kind of sum will feature prominently in the analysis of our Skorohod-Stratonovich corrections.

% We first introduce the following hypothesis on the covariance function $R$ which will ease our analysis.

We state a crucial lemma about convergence of discrete sums in the second chaos of $x$. It generalizes \cite[Lemma 3.4]{LT} to a generic Gaussian process (as opposed to the fractional Brownian motion case handled in \cite{LT}).

\begin{proposition}\label{prop:bound-F} 
Let $x$ be a $\R^{d}$-valued Gaussian process satisfying  Hypotheses \ref{hyp:2d-var-R} and ~\ref{hyp:var-R}.
For $n\ge 1$ we consider the uniform partition on $\ott$, namely $t_k=\frac kn T$. We define a process $F=\{F_{t}^{ij}; t\in  \ll 0, T\rr, i,j=1,\ldots,d  \}$ by $F_0^{ij}=0$, and for all $t>0$,
\begin{equation}\label{eq:def-F}
F_{t}^{ij}=\sum\limits_{t_k=0}^{t_-} \Big(\bx_{t_kt_{k+1}}^{2;ij}-\be[\bx_{t_kt_{k+1}}^{2;ij}]\Big)= \begin{cases}
\sum\limits_{t_k=0}^{t_-}\bx_{t_kt_{k+1}}^{2;ij}~,& i\neq j,\\
\sum\limits_{t_k=0}^{t_-}\Big(\bx_{t_kt_{k+1}}^{2;ii}-\be[\bx_{t_kt_{k+1}}^{2;ii}]\Big)~, & i=j,
\end{cases}
\end{equation}
where we recall the notation $t_-$ from Section \ref{sec:notation}  and where $\bx^2$ is introduced in Definition~\ref{def:RP}.  Then for all $q\ge 1$, $\rho'\in(\rho,2)$, $(s,t)\in \cs_{2}(\ll 0, T\rr)$ and $n\ge 1$ the following inequality holds true
\begin{equation}\label{e:estimation-dF}
\left(\be\lc\left|\delta F_{st}^{ij}\right|^q\rc\right)^{1/q}\le
C \frac{(t-s)^{\frac12}}{n^{\beta-\frac12}}\,, 
\end{equation}
 where  $C=C(q,\rho,T),\beta=\frac1\rho\in(1/2,1]$  for  $i\neq j$,  and   $C=C(q,\rho, \rho',T), \beta=\frac1{\rho'}\in (1/2,1)$  for  $i=j$.

\end{proposition}

\begin{proof} Due to the hyper-contractivity property of the second Wiener chaos, it suffices to show the case $q=2$. In addition, we assume (without loss of generality) that $s=t_{m_1} <t=t_{m_2}$ for $0\le m_1<m_2\le n$. 

\noindent
\textit{Case 1: $i=j$.} In this case, due to the definition \eqref{eq:def-F} of $F$ and the geometric nature of $\bx$ assumed in Definition \ref{def:RP}, we have
\begin{equation*}
\be\lc\left(\delta F_{st}^{ii}\right)^2\rc=
\be\left(\sum_{k=m_1}^{m_2-1}  \left[(\bx^{\1;i}_{t_kt_{k+1}})^2- \be[(\bx^{\1;i}_{t_kt_{k+1}})^2]\right]\right)^2,
\end{equation*}
and expanding the square on the right hand side above we get
\begin{equation}\label{a2}
\be\lc\left(\delta F_{st}^{ii}\right)^2\rc=
\sum_{k,l=m_1}^{m_2-1} \left\{ \be[(\bx^{\1;i}_{t_kt_{k+1}})^2(\bx^{\1;i}_{t_lt_{l+1}})^2]- \be[(\bx^{\1;i}_{t_kt_{k+1}})^2]\be[(\bx^{\1;i}_{t_lt_{l+1}})^2]\right\}.
\end{equation}
In order to evaluate the right-hand side of \eqref{a2} we apply a particular case of Wick's formula for centered Gaussian random variables $X$ and $Y$, which can be stated as:
\[
\be[X^{2} Y^{2}] - \be[X^{2}]\, \be[Y^{2}]
= 2 \lp \be[X \, Y] \rp^{2}. 
\]
Plugging this result into \eqref{a2} and recalling the definition \eqref{eq:rect-increment-cov-fct} of $R_{st}^{uv}$ we obtain
\begin{equation}\label{a21}
\be\lc\left(\delta F_{st}^{ii}\right)^2\rc=
2 \sum_{k,l=m_1}^{m_2-1} 
\left( \be \left[ \bx^{\1;i}_{t_kt_{k+1}} \bx^{\1;i}_{t_lt_{l+1}}\right]\right)^2
=2 \sum_{k,l=m_1}^{m_2-1} \lp R^{t_kt_{k+1}}_{t_lt_{l+1}} \rp^{2}.
\end{equation}
Therefore invoking elementary properties of $p$-variations we end up with
\begin{align}\label{e:delta-F2}
\be\lc\left(\delta F_{st}^{ii}\right)^2\rc
\le&
2\sup_{k,l} \left( R^{t_k t_{k+1}}_{t_l t_{l+1}} \right)^{2-\rho} \|R\|^\rho_{\rho-\text{var};[s,t]^2}.
\end{align}
  On the right-hand side of \eqref{e:delta-F2}, notice that  under Hypothesis \ref{hyp:var-R}, $\|R\|^\rho_{\rho-\text{var};[s,t]^2}$ can be upper bounded by $C(t-s)$ thanks to \eqref{e:Holder-control-var-R}. Moreover, a simple use of Cauchy-Schwarz inequality, together with \eqref{eq:bound-increment-X-L2}, shows that 
\begin{align*}
|R_{t_lt_{l+1}}^{t_k t_{k+1}}| =\left|\E[\bx^{1,i}_{t_lt_{l+1}}\bx^{1,i}_{t_kt_{k+1}}]\right|\le \left( \E\left[\left|\bx^{1,i}_{t_lt_{l+1}}\right|^2\right]\E\left[\left|\bx^{1,i}_{t_kt_{k+1}}\right|^2\right] \right)^\frac12\le \frac{C_T}{n^{\frac1\rho}}.
\end{align*}
Reporting this information into \eqref{e:delta-F2} and recalling that $\beta=\frac1\rho$, it is seen that
\begin{equation}\label{e:estimation-dF1}
\be\lc\left(\delta F_{st}^{ii}\right)^2\rc
\le
C T^{2\beta-1}\frac{(t-s)}{n^{2\beta-1}}.
\end{equation}
 This ends our proof for the case $i=j$.

\noindent
\textit{Case 2: $i\ne j$.}
According to our definition \eqref{eq:def-F}, if $i\ne j$ we have
\begin{equation*}
\be\lc\left(\delta F_{st}^{ij}\right)^2\rc
=
\be\lc\left(\sum\limits_{k=m_1}^{m_2-1} \bx_{t_{k}t_{k+1}}^{\2;ij} \right)^2 \rc  
=
\sum\limits_{k,l=m_1}^{m_2-1} \be\lc  \bx_{t_{k}t_{k+1}}^{\2;ij} \bx_{t_{l}t_{l+1}}^{\2;ij}  \rc.
\end{equation*}
Therefore, invoking the proofs of \cite[Theorem 15.33 and Proposition 15.28]{FV-bk} for the computation of $\be[  \bx_{t_{k}t_{k+1}}^{\2;ij} \bx_{t_{l}t_{l+1}}^{\2;ij}  ]$, we end up with
\begin{equation}\label{a3}
\be\lc\left(\delta F_{st}^{ij}\right)^2\rc
=
\sum_{k,l=m_1}^{m_2-1} \int_{t_k}^{t_{k+1}}\int_{t_l}^{t_{l+1}} 
R_{t_{k}v_{1}}^{t_{l}v_{2}} \, \dd R(v_1, v_2).
\end{equation}
We now fix $(k,l)$ and denote $G(v_1, v_2)=R_{t_{k}v_{1}}^{t_{l}v_{2}}$. Then $G(t_{k},\cdot)=G(\cdot, t_l)=0$. For any  $\rho'\in(\rho, 2)$, Hypothesis \ref{hyp:2d-var-R} implies $R$ has finite 2d $\rho'$-variation, and  Hypothesis \ref{hyp:var-R}  implies both $R(t, \cdot)$ and $R(\cdot, t)$ have finite $\rho'$-variation for all $t\in[0,T]$. Hence resorting to Theorem ~\ref{2d-Young}, we have for some fixed $\rho'\in(\rho, 2),$
 \begin{equation*}
\left| \int_{t_k}^{t_{k+1}}\int_{t_l}^{t_{l+1}} 
 R_{t_{k}v_{1}}^{t_{l}v_{2}} \, \dd R(v_1, v_2) \right|
 \le 
 C\|R\|^2_{ \rho'\text{-var};[t_k, t_{k+1}]\times[t_l, t_{l+1}]},
 \end{equation*}
for some constant $C=C(\rho', T)$ depending on $(\rho',T)$ only. Plugging this inequality into \eqref{a3}    we obtain 
 \begin{align*}
\be\lc\left(\delta F_{st}^{ij}\right)^2\rc
\le
&C(\rho', T)\sum_{k,l=m_1}^{m_2-1}\|R\|^2_{\rho'\text{-var};[t_k, t_{k+1}]\times[t_l, t_{l+1}]}\notag\\
&\le C(\rho', T)  \sup_{k,l} \|R\|^{2-\rho'}_{\rho\text{-var};[t_k, t_{k+1}]\times[t_l, t_{l+1}]}  \sum_{k,l=m_1}^{m_2-1}\|R\|^{\rho'}_{\rho'\text{-var};[t_k, t_{k+1}]\times[t_l, t_{l+1}]}. 
\end{align*}
Therefore, thanks to Remark ~\ref{remark:control-2d-variation},  we have
\[\be\lc\left(\delta F_{st}^{ij}\right)^2\rc\le C(\rho', T)\sup_{k,l} \|R\|^{2-\rho'}_{\rho'\text{-var};[t_k, t_{k+1}]\times[t_l, t_{l+1}]}  \sum_{k,l=m_1}^{m_2-1}\omega([t_k, t_{k+1}]\times[t_l, t_{l+1}]),\]
where $\omega$ is a control given in \eqref{e:w}. 
Furthermore,  the super-additivity of $\omega$ yields
\begin{align}
\be\lc\left(\delta F_{st}^{ij}\right)^2\rc
&\le  C(\rho', T)\sup_{k,l} \|R\|^{2-\rho'}_{\rho'\text{-var};[t_k, t_{k+1}]\times[t_l, t_{l+1}]}  ~ \omega([s,t]^2)\notag\\
&\le  C(\rho, \rho', T) \sup_{k,l} \|R\|^{2-\rho'}_{\rho'\text{-var};[t_k, t_{k+1}]\times[t_l, t_{l+1}]} ~ (t-s),\label{e:estimation-dF2}
\end{align}
  where the last inequality is due to \eqref{e:w}, \eqref{e:norms-compare}, and \eqref{e:bound-R-2rho}.   Finally, by Hypothesis 
~\ref{hyp:var-R} (and Remark \ref{remark:holder-R''}),  we have \[\|R\|^{2\rho'}_{\rho'\text{-var};[s,t]\times [u,v]}\le C(\rho, \rho', T)(t-s)(u-v),\]
and therefore  setting $\beta=1/\rho'$, inequality  \eqref{e:estimation-dF2} becomes 
 \begin{align}\label{e:estimation-dF3}
\be\lc\left(\delta F_{st}^{ij}\right)^2\rc
\le
&C(\rho, \rho', T) \left(\frac Tn\right)^{2\beta-1}(t-s).
\end{align}

With \eqref{e:estimation-dF1} and \eqref{e:estimation-dF3} in hand our claim \eqref{e:estimation-dF} is now easily achieved,  which concludes the proof.
\end{proof}

   Note that \eqref{e:estimation-dF} is still valid  for both cases of  $i=j$ and $i\neq j$, if we choose $\beta=\frac1{\rho'}$ for any $\rho'\in(\rho, 2)$.  We now give a weighted version of Proposition \ref{prop:bound-F}, which plays an important role in our correction computations.
\begin{proposition}\label{prop:weighted-sum}
Let $x$ be a $\R^{d}$-valued Gaussian process satisfying Hypotheses ~\ref{hyp:2d-var-R} and ~ \ref{hyp:var-R}.
 Let $\rho'\in(\rho, 2)$ be fixed. For $n\ge 1$ we consider the uniform partition on $\ott$, namely $t_k=\frac kn T$, as well as the process $F$ defined by \eqref{eq:def-F}.
Let now $f$ be a controlled process in the $L^{q}(\oom)$ sense, namely such that there exists a process $g$ fulfilling (in the matrix sense), for some $\gamma\in(\frac14, \frac1{2\rho})$ and for all $q\ge1$,
\begin{equation}\label{eq:controlled-f}
\|f_{t}\|_{q} + \|g_{t}\|_{q} \leq C,
\quad
\|\delta f_{st}  - g_{s} \, \bx^{\1}_{st}\|_{q} \leq C(t-s)^{2\ga},
\quad
\|\delta g_{st} \|_{q} \leq C (t-s)^{\ga}.
\end{equation}
Then   the following estimate holds true for $(s,t) \in \cs_{2}(\ll 0,T \rr)$:
\begin{eqnarray*}
\Big\| \sum_{t_{k}=s}^{t-} f_{t_{k}} \otimes   \delta F_{t_{k}t_{k+1}}\Big\|_{q} &\leq& C\frac{(t-s)^{\frac12}}{n^{\beta-\frac12}},
\end{eqnarray*}
 where   $C=C(q,\rho, \rho',T)$ and $\beta=\frac1{\rho'}\in (1/2,1)$.
\end{proposition}
 
\begin{proof}
This proposition was proved in \cite[Corollary 4.9]{LT} when $x$ is a fractional Brownian motion. Although we generalize this result to a wider class of Gaussian processes, our proof goes along the same lines. Therefore we shall omit the details for sake of conciseness. 
%
%Let   $h_{st}=n^{\beta-\frac12} \delta F_{st}$ and $\alpha=\frac12$, and then Proposition \ref{prop:bound-F} allows us to apply (4.14) of   \cite[Proposition 4.7]{LT} to get the desired result, noting that $f_{t_k}=f_s+\delta f_{st_k}$. 
\end{proof}

%As a direct consequence of Proposition \ref{prop:weighted-sum}, we have the following estimate,  which will be used to estimate our Stratonovich-Skorohod correction terms.
%\begin{corollary}\label{cor:bound-F}
%Let $x$ be a $\R^{d}$-valued Gaussian process satisfying {\blue Hypotheses ~\ref{hyp:2d-var-R} and ~\ref{hyp:var-R}}.
%For $n\ge 1$ we consider the uniform partition on $\ott$, namely $t_k=\frac kn T$, as well as the process $F$ defined by \eqref{eq:def-F}.  Let $\eta$ be a continuous stochastic process on $[0,T]$ such that $\sup_{t\in[0,T]}\E|\eta_t|^q<\infty$  and $\|\delta\eta_{st}\|_q\le C (t-s)^{2\gamma}$ for some $\gamma\in(\frac14, \frac1{2\rho})$ and for all $q\ge1$. Then there exists a constant $C=C(q,T)$ such that  for all $(s,t) \in \cs_{2}(\ll 0,T \rr)$
%\begin{eqnarray*}
%\Big\| \sum_{t_{k}=s}^{t-} \eta_{t_{k}} \otimes   \delta F_{t_{k}t_{k+1}}\Big\|_{q} &\leq& C\frac{(t-s)^{\frac12}}{n^{\beta-\frac12}}\,.
%\end{eqnarray*}
% \end{corollary}
%\begin{proof}
%We just apply Proposition \ref{prop:weighted-sum} to a process $\eta=f$ satisfying \eqref{eq:controlled-f} with $g=0$. 
%\end{proof}

\section{Correction terms in the case $2\le p <3$}

In this section we derive a  correction formula for controlled processes which are also in the domain of the Skorohod integral. As mentioned in the introduction, we have restricted our analysis to the case $p<3$. Although we believe that our methodology could be extended to $p<4$, this generalisation would require a cumbersome study of third order integrals and related weighted sums. 

\begin{theorem}\label{thm}
Let $x$ be a Gaussian rough path with covariance given by \eqref{eq:def-covariance-X} satisfying Hypotheses \ref{hyp:2d-var-R} and \ref{hyp:var-R}  with $\rho\in[1,\frac32)$. This implies that $x$ has finite $p$-variation for $p>2\rho$. We can assume  $\frac1p+\frac1\rho>1$, noting that $\rho<\frac32.$ 

% Assume that $R^{st}_{uv}\ge0$ for  all $[u,v]\subset[s,t]\subset [0,T]$ such that $|t-s|\le h$  for some $h>0$.   We also assume that $|\mu|(dx)$ is absolutely continuous with respect to the Lebesgue measure, where $\mu:=\frac{\partial^2 R}{\partial t\partial s}$ is a Radon  measure on $(0,T)^2\backslash [(x,y)\in(0,T)^2: x=y]$. 
 
  Let $y$ be a second-order controlled process in the sense of Definition \ref{def:ctrld-process'}, and we assume  $\E[\|y\|^2_{p-{\rm var};[0,T]}]<\infty$. In particular, the rough integral $\int_0^t y_r d{\bf x}_r$ is defined as in Proposition~\ref{prop:intg-ctrld-process},  resorting to the convention on inner products of Section \ref{sec:notation}.  We also assume that $y\in\D^{1,2}(\ch^{d})$, so that the Skorohod integral of $y$ given in \eqref{e:divergence} is well defined.   Furthermore, we suppose that  $\mathbf{D}_0 y$ has finite $p$-variation with $\E[\|\mathbf{D}_0y\|^2_{p-{\rm var};[0,T]}]<\infty$, and  $\mathbf{D}y$ has finite 2d $p$-variation with $\E[\|\mathbf{D}y\|^2_{p-{\rm var};[0,T]^2}]<\infty.$ Then for all $t\in\ott$ we have almost surely
\begin{equation}\label{eq:strato-sko-first}
\int_{0}^{t} y_r \, \dd \bx_{r}
=
\int_{0}^{t} y_r \ddi x_r 
+\frac12 \sum_{i=1}^d \int_{0}^{t} y_{r}^{x;ii} \dd R_r 
+\sum_{i=1}^d  \int_{\cs_{2}([0,t])} \left(  \bd^i_{r_{1}} y^i_{r_{2}} -y_{r_{2}}^{x;ii}\right)\dd R(r_{1},r_{2}),
\end{equation}
where we recall from Section \ref{sec:rough-path-above-X}  that $R_r:=R(r,r)$ and where the Malliavin derivative $\mathbf D^i$ is introduced in Remark \ref{remark:DF}.
\end{theorem}

\begin{proof}
Let $\pi=\pi_{n}$ be the uniform partition of order $n$ of $[0,t]$, whose generic element is still denoted by $t_k=\frac knt.$ A natural discretization of $y$ along $\pi$ is given by 
\begin{equation}\label{eq:y-pi}
 y^\pi(r)=\sum_{k=0}^{n-1} y_{t_k} \1_{[t_k, t_{k+1})}(r), \quad r\in[0,t].
\end{equation}
Notice that we have assumed that  $y\in \D^{1,2}(\ch^d)$. Hence both divergence integrals $\delta^\diamond(y^\pi)$ and $\delta^\diamond(y)$, as given in \eqref{e:divergence}, are well defined. Moreover,  according to \eqref{e:relation}, we have 
\[\int_0^t y_r^\pi \dd^\diamond x_r=\sum_{i=1}^d\sum_{k=0}^{n-1} y_{t_k}^{i}\diamond \mathbf{x}_{t_kt_{k+1}}^{\1;i},\]
and owing to \eqref{e:relation'} this can be recast as 
\begin{equation}\label{e:discrete-int}
\int_0^t y_r^\pi \dd^\diamond x_r =\sum_{i=1}^d \sum_{k=0}^{n-1} y_{t_k}^i \mathbf{x}^{1;i}_{t_kt_{k+1}} -\langle \mathbf{D}^i y_{t_k}^i , \1_{[t_k, t_{k+1}]}\rangle_{\mathcal H} \, .
\end{equation}
In addition, we will prove in the forthcoming Lemma \ref{lem:con-y}  that  $\delta^\diamond(y^\pi)$ converges  in $L^2(\Omega)$ to $\delta^\diamond(y)$. Otherwise stated, for $t\in\ott$ we have
\begin{equation}\label{e:discrete-int'}
\int_{0}^{t} y_r \, \ddi x_r
= \lim_{n\to\infty}\int_{0}^{t} y_r^{\pi} \, \ddi x_r.
\end{equation}
Therefore  combining \eqref{e:discrete-int} and \eqref{e:discrete-int'}, we get the following limit in $L^2(\Omega)$:
\begin{eqnarray}\label{b1}
\int_{0}^{t} y_r \, \ddi x_r
&=&  \lim_{n\to\infty}\sum_{i=1}^d \sum_{k=0}^{n-1} \left(y_{t_k}^i \bx^{\1;i}_{t_kt_{k+1}}
-\langle {\mathbf D}^i y^i_{t_k}, \1_{[t_k, t_{k+1}]} \rangle_\ch \right). 
\end{eqnarray}
On the other hand, owing to the fact that $y$ is a controlled process in the sense of Definition~\ref{def:ctrld-process}, Proposition \ref{prop:intg-ctrld-process} asserts that $\int_{0}^{t} y_r \dd \bx_r$ is defined as a rough paths integral and  hence almost surely we have
\begin{equation}\label{b2}
\int_{0}^{t} y_r \dd \bx_r
=
\lim_{n\to\infty}\left(\sum_{i=1}^d \sum_{k=0}^{n-1} y_{t_k}^i \bx^{\1;i}_{t_kt_{k+1}}
+ \sum_{i,j=1}^d \sum_{k=0}^{n-1} y_{t_k}^{x;ij} \bx^{\2;ij}_{t_kt_{k+1}}\right).
\end{equation}
Gathering relations \eqref{b1} and \eqref{b2}, we get the following expression for the Stratonovich-Skorohod correction term:
\begin{align}\label{b3}
\int_{0}^{t} y_r \dd \bx_r -\int_{0}^{t} y_r \, \ddi x_r 
=  \lim_{n\to\infty}\sum_{i=1}^d\sum_{k=0}^{n-1} 
\left(  \sum_{j=1}^d y_{t_k}^{x;ij} \bx^{2;ij}_{t_kt_{k+1}} 
+ \langle {\mathbf D}^iy^i_{t_k}, \1_{[t_k, t_{k+1}]} \rangle_\ch  \right),
\end{align}
 where the limit on the right-hand side above is understood in probability.
In \eqref{b3}, notice that the left-hand side is well defined thanks to the standing assumptions of our Theorem. Hence the right-hand side of \eqref{b3} also makes sense, and we will now identify the limits therein.

In order to compute the limit for the terms $y_{t_k}^{x;ij} \bx^{2;ij}_{t_kt_{k+1}}$ in \eqref{b3}, observe that $y$ is a second order controlled process according to Definition \ref{def:ctrld-process'}. Hence $y^x$ is a controlled process satisfying relation \eqref{e:y-2}. Since we have assumed that Hypotheses \ref{hyp:2d-var-R} and \ref{hyp:var-R} are fulfilled, Proposition \ref{prop:weighted-sum} for the increment $F$ can be applied with $f=y^x$. Recalling (see \eqref{eq:def-F}) that 
\[\delta F_{t_kt_{k+1}}^{ij}=\bx_{t_kt_{k+1}}^{2;ij}-\E[\bx_{t_kt_{k+1}}^{2;ij}],\]
we end up with the following relation, valid for $i,j=1,\dots, d$,  where the limit has to be considered  in the $L^1(\Omega)$ sense:
\begin{equation}\label{e:60}
\lim\limits_{n\to\infty}  \sum_{k=0}^{n-1}
y_{t_k}^{x;ij} \Big(\bx^{2;ij}_{t_kt_{k+1}}-\be \lc \bx^{2;ij}_{t_kt_{k+1}} \rc\Big)
=0.
\end{equation}
In particular, going back to \eqref{b3}, we get that for $i\ne j$ we have
\begin{equation}\label{e:60'}
\lim\limits_{n\to\infty}  \sum\limits_{i\neq j}^d\sum_{k=0}^{n-1}
y_{t_k}^{x;ij} \bx^{2;ij}_{t_kt_{k+1}}
= 0.
\end{equation}

Let us deal with the left-hand side of \eqref{e:60} when $i=j$. Specifically, we will express the limit of the sums $ \sum_{k=0}^{n-1}
y_{t_k}^{x;ii} \be \lc \bx^{2;ii}_{t_kt_{k+1}} \rc$ as a Young integral. To this aim, notice that $2\bx_{t_kt_{k+1}}^{2,ii}=\left(\bx_{t_kt_{k+1}}^{1,i}\right)^2$ due to the geometric assumption in Definition \ref{def:RP}. Hence invoking the fact that $R_{t_k}=R(t_k, t_k)$ we have
\begin{eqnarray*}
 2\be \lc \bx^{2;ii}_{t_kt_{k+1}} \rc
&=&\be\lc(x_{t_{k+1}}^i-x_{t_{k}}^i)^2\rc
= R_{t_{k+1}}-2R(t_{k+1}, t_k)+R_{t_k}\\
&=&\lp R_{t_{k+1}}-R_{t_k}\rp - 2 \lp R(t_{k+1}, t_{k})-R(t_{k}, t_{k})\rp.
\end{eqnarray*}
Therefore for all $i=1, \dots, d$, we obtain a decomposition of the form
\begin{equation}\label{e:61}
\sum_{k=0}^{n-1}
y_{t_k}^{x;ii} \be \lc \bx^{2;ii}_{t_kt_{k+1}} \rc=\frac12 I_n^i -J_n^i
\end{equation}
where $I_n^i, J_n^i$ are respectively defined by 
\begin{equation}\label{e:62}
I_n^i =\sum_{k=0}^{n-1} y_{t_k}^{x;ii}\delta R_{t_kt_{k+1}}, 
\quad\text{and }\quad 
J_n^i =\sum_{k=0}^{n-1} y_{t_k}^{x;ii} \big(R(t_{k+1}, t_k)-R(t_k, t_k)\big).
\end{equation}

The limit of for the term $I^i_n$ in \eqref{e:61} can be computed easily. Indeed, thanks to Remark~\ref{remark:holder-R} we know that $t\to R_t$ has finite $\rho$-variation. Furthermore, since $y$ is a second order controlled process, Definition \ref{def:ctrld-process'} entails that $y^x$ has finite $p$-variation. We have also mentioned in Theorem \ref{thm} that $p^{-1}+\rho^{-1}>1$. Hence classical Young integration arguments reveal that for $i=1, \dots, d$ we have  almost surely,
\begin{equation}\label{e:63}
\lim_{n\to \infty} I_n^i =\int_0^t y_r^{x;ii}dR_r .
\end{equation}
%{\blue Moreover, notice that
%\begin{align*}
%|I_n^i|\le \sum_{k=0}^{n-1} 
%\left| y_{t_k}^{x;ii}\delta R_{t_kt_{k+1}}\right|\le \sum_{k=0}^{n-1} |y^{x;ii}_{t_k}| \|R_\cdot\|_{\rho-{\rm var};[t_k, t_{k+1}]}. 
%\end{align*}
%For a fixed small $\e>0$, let $\alpha, \alpha'$ be the numbers such that $\alpha=2+\e$ and $\alpha^{-1}+(\alpha')^{-1}=1$. Then, H\"older inequality yields
%\begin{align*}
%|I_n^i|^{\alpha}\le &\sum_{k=0}^{n-1} |y_{t_k}^{x;ii}|^{\alpha} \|R_\cdot\|_{\rho-{\rm var};[t_k, t_{k+1}]} \left(\sum_{k=0}^{n-1} \|R_\cdot\|_{\rho-{\rm var};[t_k, t_{k+1}]}\right)^{\frac{\alpha'}{\alpha}}\\
%\le& \sum_{k=0}^{n-1} |y_{t_k}^{x;ii}|^{\alpha} \|R_\cdot\|_{\rho-{\rm var};[t_k, t_{k+1}]} \left( \|R_\cdot\|_{\rho-{\rm var};[0, t]}\right)^{\frac{\alpha'}{\alpha}}
%\end{align*}
% {\red If we assume $t\to \E[\|y^x_t\|^{\alpha}]$ is continuous and $\int_0^T \E[\|y_s^x\|^{\alpha}] d \|R_\cdot\|_{\rho-{\rm var};[0, t]} <\infty$,} then we have $\sup_{n}\E[|I_n^i|^\alpha] <\infty$, and hence  the limit in \eqref{e:63} also holds in $L^2(\Omega)$ by dominated convergence theorem.}

As far as the term $J_n^i$ in \eqref{e:62} is concerned, let us recast this expression in terms of a 2-d Riemann sum. Namely we define another uniform partition $\{v_l; 0\le l\le n-1\}$ of $[0,t]$, with $v_l=\frac ln t.$ Then we start by writing 
\begin{equation}\label{e:64}
J_n^i=\sum_{k=0}^{n-1} y_{t_k}^{x;ii}\big(R(t_{k+1},v_k)-R(t_{k},v_k) \big).
\end{equation}
In addition, notice that thanks to Remark \ref{remark:holder-R} we have $R(\cdot, 0)=0$. Thus an immediate telescoping sum argument yields the following relation, valid for $k=0, \dots, n-1:$
\[R(t_{k+1}, v_k)-R(t_k, v_k)=\sum_{l=0}^{k-1} R_{v_l v_{l+1}}^{t_k t_{k+1}}.\]
Reporting this identity into \eqref{e:64}, we get 
\begin{equation}\label{e:65'}
J_n^i=\sum_{k=0}^{n-1} y_{t_k}^{x;ii}\sum_{l=0}^{k-1} R_{v_lv_{l+1}}^{t_kt_{k+1}}=\sum_{0\le l<k\le n-1} y_{t_k}^{x;ii} R_{v_lv_{l+1}}^{t_kt_{k+1}}.
\end{equation}
This decomposition prompts us to define a degenerate function $f$ in the plane as $f^{i}(u,v)=y_u^{x;ii}\1_{[0<v<u<t]}$. With this notation in hand, relation \eqref{e:65'} reads
\[J_n^i=\sum_{k,l=0}^{n-1} f^i(t_k, v_l) R_{v_lv_{l+1}}^{t_kt_{k+1}}. \]

In order to analyze the convergence of $J_n^i$, we now argue as follows: first $R$ has a finite 2-dimensional $\rho$-variation. The function $f^i(u,v)=y_u^{x,ii}\1_{[0<v<u<t]}$ is also easily seen to have a finite 2-dimensional $p$-variation  (owing to the fact that $y^{x,ii}$ has finite $p$-variation), and recall that $p^{-1}+\rho^{-1}>1$. Hence standard convergence procedures for 2d-Young integrals show that  almost surely
\begin{equation}\label{e:65}
\lim_{n\to\infty} J_n^i =\int_0^t\int_0^t f^i(u,v) \, \dd R(u, v)=\int_{\cs_{2}([0,t])} y_{r_{2}}^{x;ii} \, \dd R(r_{1},r_{2}). 
\end{equation}
%{\blue Similar to \eqref{e:63}, {\red If we assume that for some $\e>0$,  $t\to \E[\|y^x_t\|^{2+\e}]$ is continuous and $ \int_{\cs_{2}([0,T])}  \E[\|y_{r_2}^x\|^{2+\e}] d \|R\|_{\rho-{\rm var};[0, r_1]\times[0,r_2]} <\infty$,}  the limit in \eqref{e:65} also holds in $L^2(\Omega)$. }

Summarizing our considerations for the case $i=j$, we gather \eqref{e:63} and \eqref{e:65} into the decomposition \eqref{e:61}. We conclude that  almost surely, 
\begin{equation}\label{b4}
\lim\limits_{n\to\infty}   \sum\limits_{i=1}^d\sum_{k=0}^{n-1}
y_{t_k}^{x;ii} \E[ \bx^{2;ii}_{t_kt_{k+1}}]
=
\frac12 \sum_{i=1}^d \int_0^t y_r^{x;ii} \dd R_r- \sum_{i=1}^d
\int_{\cs_{2}([0,t])} y_{r_{2}}^{x;ii} \, \dd R(r_{1},r_{2}).
\end{equation}

We now go back to \eqref{b3}, and handle the terms $\langle \mathbf D^iy^i_{t_k}, \1_{[t_k, t_{k+1}]} \rangle_\ch$ therein. We write the inner product in $\ch$ in an explicit way thanks to \eqref{rep H norm}, which yields
\begin{equation*}
\langle \mathbf D^iy^i_{t_k}, \1_{[t_k, t_{k+1}]} \rangle_\ch
=
\int_0^t\int_0^t \mathbf D_{r_{1}}^{i} y^i_{t_k}\1_{[0,t_k]} (r_{1})  \1_{[t_k, t_{k+1}]}(r_{2}) \, \dd R(r_{1},r_{2}).
\end{equation*}
We thus have
\begin{equation*}
\lim_{n\to\infty}\sum_{i=1}^d \sum_{k=0}^{n-1}  \langle \mathbf D^iy^i_{t_k}, \1_{[t_k, t_{k+1}]} \rangle_\ch
=
\lim_{n\to\infty}\sum_{i=1}^d \sum_{k=0}^{n-1}
\int_0^t\int_0^t \mathbf D_{r_{1}}^{i} y^i_{t_k}\1_{[0,t_k]} (r_{1})  \1_{[t_k, t_{k+1}]}(r_{2}) \, \dd R(r_{1},r_{2}). 
\end{equation*}
We now argue similarly to what we did for \eqref{e:65}. Namely one of our standing assumptions is that $(r_1, r_2)\to \mathbf D^i_{r_1} y_{r_2}^i \1_{\mathcal S_2}(r_1, r_2)$ has a finite 2-dimensional $p$-variation. Since $R$ admits a finite $\rho$-variation and $p^{-1}+\rho^{-1}>1$, standard results concerning convergence of Riemann sums to Young integrals  show that  almost surely we have
\begin{equation}\label{b5}
\lim_{n\to\infty}\sum_{i=1}^d \sum_{k=0}^{n-1}  \langle \mathbf D^iy^i_{t_k}, \1_{[t_k, t_{k+1}]} \rangle_\ch
=\sum_{i=1}^d
\int_{\cs_{2}([0,t])}   \mathbf D^i_{r_{1}} y^i_{r_{2}} \, \dd R(r_{1},r_{2}).
\end{equation}
%{\blue Using a similar argument for \eqref{e:63}, {\red if we assume that for some $\e>0$,  $(s,t) \to \E[\|\mathbf D_sy^x_t\|^{2+\e}]$ is continuous on $[0<s<t<T]$ and $ \int_{\cs_{2}([0,T])}  \E[\|\mathbf D_{r_1}y_{r_2}^x\|^{2+\e}] d \|R\|_{\rho-{\rm var};[0, r_1]\times[0,r_2]} <\infty$, } we can show that the limit in \eqref{b5} also holds in $L^2(\Omega)$.}
We can now conclude our proof easily. That is plugging \eqref{e:60}, \eqref{e:60'}, \eqref{b4} and \eqref{b5} into \eqref{b3}, we end up with,  almost surely, 
\begin{multline*}
\int_0^t y_r \dd\bx_r - \int_0^t y_r \ddi x_r \\
= 
\frac12 \sum_{i=1}^d \int_0^t y_r^{x;ii}\dd R_r -\sum_{i=1}^d \int_{\mathcal S_2([0,t])} y_{r_2}^{x;ii} \dd R(r_1, r_2)
+ \sum_{i=1}^d \int_{\mathcal S_2([0,t])} \mathbf D_{r_1}^i y_{r_2}^i \dd R(r_1, r_2),
\end{multline*}
from which the claim \eqref{eq:strato-sko-first}
is immediately deduced. This concludes the proof. 
\end{proof}

We close this section by proving a technical result which has been used in order to derive relation \eqref{e:discrete-int'}.

\begin{lemma}\label{lem:con-y}
Assume the same conditions as in Theorem \ref{thm}. Then  $y^\pi$ defined in \eqref{eq:y-pi} converges to $y$ in $\mathbb D^{1,2}(\mathcal H^{d})$, i.e. $\lim\limits_{|\pi|\to0}\E[\|y^{\pi}-y\|^2_{\ch^d}+\|\mathbf Dy^\pi- \mathbf Dy\|^2_{(\ch^d)^{\otimes 2}}]=0$.
\end{lemma}
\begin{proof}
 According to \eqref{rep H norm}, we have 
\[\|y^{\pi}-y\|_{\ch^d}^2=\sum_{i,j=0}^{n-1}\int_{[t_{i},t_{i+1}]\times[t_{j},t_{j+1}]} \langle  y_{t_{i}}- y_s, y_{t_{j}}- y_t\rangle ~\dd R(s,t),\]
where we recall that $\pi=\{0=t_0<t_1<\dots<t_n=t \}.$ On each rectangle $[t_{i}, t_{i+1}]\times[t_{j}, t_{j+1}]$ we apply Theorem \ref{2d-Young} to the function
\[f_{ij}(s,t)= \langle  y_s-y_{t_{i}}, y_t-y_{t_{j}}\rangle, \]
which is allowed since $f_{ij}$ is easily seen to be a function in $ \cac_2^{p-\text{var}}$. 

Recall that we have assumed $p^{-1}+\rho^{-1}>1$. Throughout the proof, we choose  $p'>p$ and $\rho''>\rho'>\rho$ satisfying \[(p')^{-1}+(\rho')^{-1}>1 ~\text{ and }~ (p')^{-1}+(\rho'')^{-1}>1.\]

  Since we also have $f_{ij}(t_{i}, \cdot)=0$ and $f_{ij}(\cdot, t_{j})=0$, we get 
\begin{align}\label{e:y-y'}
\|y^{\pi}-y\|_{\ch^d}^2
\le &C\sum_{i,j=0}^{n-1} \Big(\|y\|_{p'-{\rm var};[t_{i},t_{i+1}]}\|y\|_{p'-{\rm var};[t_{j},t_{j+1}]}\Big)\|R\|_{\rho'-{\rm var};[t_{i}, t_{i+1}]\times[t_{j},t_{j+1}]}.
\end{align}

In order to bound the right-hand side of \eqref{e:y-y'}, we introduce a new function $\omega_1$, defined by 
\begin{equation}\label{e:w1}
\omega_1([a,b]\times[c,d])=\|y\|^{p'}_{p'-{\rm var};[a,b]}\|y\|^{p'}_{p'-{\rm var};[c,d]}.
\end{equation}
Then it is readily checked that $\omega_1$ is also a 2d-control in the sense of Definition \ref{def:2d-control}. The following is easily deduced from~\eqref{e:y-y'}:
  \begin{align}
\|y^{\pi}-y\|_{\ch^d}^2
\le &C\sup_{i,j}\Big( \omega([t_{i}, t_{i+1}]\times[t_{j}, t_{j+1}])\Big)^{\frac1{\rho'}-\frac1{\rho''}}\notag\\
& \times \sum_{i,j=0}^{n-1} \Big(\omega_1([t_{i}, t_{i+1}]\times[t_{j}, t_{j+1}])\Big)^{\frac1{p'}} \Big( \omega([t_{i}, t_{i+1}]\times[t_{j}, t_{j+1}])\Big)^{\frac1{\rho''}},\label{e:y-y''}
\end{align}
where $\omega$ is the control defined in \eqref{e:w}.  Now both $\omega_1$ and $\omega$ above are 2d-controls. Hence an easy extension of \cite[Exercise 1.9]{FV-bk} to a 2d setting shows that $\omega_1^{1/p'}\omega^{1/{\rho''}}$ is also a 2d-control. Hence one can resort to the super-additivity property of $\omega_1^{1/p'}\omega^{1/{\rho''}}$ in order to deduce the following from \eqref{e:y-y''}:
\begin{equation}\label{e:y-y}
\|y^{\pi}-y\|_{\ch^d}^2
\le  C \sup_{i,j}\Big( \omega([t_{i}, t_{i+1}]\times[t_{j}, t_{j+1}])\Big)^{\frac1{\rho'}-\frac1{\rho''}}   \Big(\omega_1([0,t]^2)\Big)^{\frac1{p'}} \Big( \omega([0, t]^2)\Big)^{\frac1{\rho''}}.
\end{equation}

We now turn to an upper bound on  $\|\mathbf Dy^\pi- \mathbf Dy\|_{(\ch^d)^{\otimes 2}}$. To this aim we first express this quantity using the norm in $(\ch^d)^{\otimes2}$ induced by \eqref{rep H norm}.  This yields
\begin{align}
&\|\mathbf Dy^\pi- \mathbf Dy\|^2_{(\ch^d)^{\otimes 2}}\notag\\
=&\sum_{i,j=0}^{n-1}\int_{[0,t_{i+1}]\times[0,t_{j+1}]\times [t_{i},t_{i+1}]\times[t_{j},t_{j+1}]} \langle  \mathbf{D}_u y_{t_{i}}-\mathbf{D}_u y_s, \mathbf{D}_v y_{t_{j}}-\mathbf{D}_v y_t\rangle ~ \dd R(u,v) \dd R(s,t). \label{e:dy-dy1}
\end{align}
We apply Lemma \ref{4d-Young} to the right-hand side of \eqref{e:dy-dy1} and get
\begin{align*}
&\|\mathbf Dy^\pi- \mathbf Dy\|^2_{(\ch^d)^{\otimes 2}}
\le  C \sum_{i,j=0}^{n-1}  \|R\|_{\rho'-{\rm var};[0, t_{i+1}]\times[0,t_{j+1}]}\|R\|_{\rho'-{\rm var};[t_{i}, t_{i+1}]\times[t_{j},t_{j+1}]}\notag\\
&\qquad \qquad\qquad\times \Big( \|\mathbf{D}_0y_{t_{i}} -\mathbf{D}_0y_{\cdot}  \|_{p'-{\rm var};[t_{i},t_{i+1}]} +\|\mathbf{D}_\cdot y_{t_{i}} -\mathbf{D}_\cdot y_\cdot \|_{p'-{\rm var};[0,t_{i+1}]\times[t_i,t_{i+1}]}\Big)\notag\\
&\qquad \qquad\qquad\times \Big( \|\mathbf{D}_0y_{t_{j}} -\mathbf{D}_0y_{\cdot}  \|_{p'-{\rm var};[t_{j},t_{j+1}]} +\|\mathbf{D}_\cdot y_{t_{j}} -\mathbf{D}_\cdot y_\cdot \|_{p'-{\rm var};[0,t_{j+1}]\times[t_{j},t_{j+1}]}\Big). 
\end{align*}
As a preliminary step, we also bound the variations on intervals of the form $[0,t_j]$ by variations on $[0,T]$. 
Thus one can bound $\|\mathbf Dy^\pi- \mathbf Dy\|^2_{(\ch^d)^{\otimes 2}}$ by
\begin{multline}\label{e:dy-dy'}
C \|R\|_{\rho'-{\rm var};[0, T]^2}\sum_{i,j=0}^{n-1}  \|R\|_{\rho'-{\rm var};[t_{i}, t_{i+1}]\times[t_{j},t_{j+1}]} 
\Big( \|\mathbf{D}_0y \|_{p'-{\rm var};[t_{i},t_{i+1}]}  
+\|\mathbf{D} y \|_{p'-{\rm var};[0,T]\times[t_i,t_{i+1}]}\Big) \\
\times \Big( \|\mathbf{D}_0y \|_{p'-{\rm var};[t_{j},t_{j+1}]} +\|\mathbf{D} y\|_{p'-{\rm var};[0,T]\times[t_{j},t_{j+1}]}\Big).
\end{multline}

We now wish to apply super-additivity properties of the $p$-variations, as we did for \eqref{e:y-y}. However, note that  the function $[a,b]\times [c,d]\mapsto \|\mathbf{D} y \|^{p'}_{p'-{\rm var};[a,b]\times[c,d]}$ may fail to be super-additive (see \cite[Theorem 1]{fv11}). Hence we need to resort to the controlled 2d variation as introduced in Definition \ref{def:controlled-variation}.  Specifically, it follows from \eqref{e:dy-dy'} that $\|\mathbf Dy^\pi- \mathbf Dy\|^2_{(\ch^d)^{\otimes 2}}$ can be upper bounded by
\begin{multline}
C \|R\|_{\rho'-{\rm var};[0, T]^2}\sum_{i,j=0}^{n-1}  \|R\|_{\rho'-{\rm var};[t_{i}, t_{i+1}]\times[t_{j},t_{j+1}]} 
\Big( \|\mathbf{D}_0y \|_{p'-{\rm var};[t_{i},t_{i+1}]} +\vvvert\mathbf{D} y \vvvert_{p'-{\rm var};[0,T]\times[t_i,t_{i+1}]}\Big)\\
 \times \Big( \|\mathbf{D}_0y \|_{p'-{\rm var};[t_{j},t_{j+1}]} +\vvvert \mathbf{D} y\vvvert_{p'-{\rm var};[0,T]\times[t_{j},t_{j+1}]}\Big). \label{e:dy-dy2}
\end{multline}
Notice that the right-hand side of \eqref{e:dy-dy2} is finite,  noting that  $p<p'$ and owing to \eqref{e:norms-compare}. Furthermore, noting that the function $[c, d] \mapsto \vvvert \mathbf Dy \vvvert^{p'}_{p'-{\rm var}; [0,T]\times[c,d]}$ is a control,  we can define the following 2d controls (where we use \cite[Exercise 1.9]{FV-bk} again):
\begin{eqnarray}
\omega_2([a,b]\times[c,d])
&=&
\|\mathbf{D}_0y \|^{p'}_{p'-{\rm var};[a,b]}\|\mathbf{D}_0y \|^{p'}_{p'-{\rm var};[c,d]} \label{e:w2}\\
\omega_3([a,b]\times[c,d])
&=& 
\vvvert\mathbf{D} y \vvvert^{p'}_{p'-{\rm var};[0,T]\times[a,b]}\vvvert\mathbf{D} y \vvvert^{p'}_{p'-{\rm var};[0,T]\times[c,d]} \label{e:w3} \\
\omega_4([a,b]\times[c,d])
&=&
\|\mathbf{D}_0y \|^{p'}_{p'-{\rm var};[a,b]}\vvvert\mathbf{D} y \vvvert^{p'}_{p'-{\rm var};[0,T]\times[c,d]}. \label{e:w4}
\end{eqnarray}
Now, similarly to \eqref{e:y-y},  relation \eqref{e:dy-dy2} entails 
 \begin{align}\label{e:Dy-Dy}
 &\|\mathbf Dy^\pi- \mathbf Dy\|^2_{(\ch^d)^{\otimes 2}} \\
& \le C \|R\|_{\rho'-{\rm var};[0, T]^2} \sup_{i,j}\Big( \omega([t_{i}, t_{i+1}]\times[t_{j}, t_{j+1}])\Big)^{\frac1{\rho'}-\frac1{\rho''}}\notag
  \Big( \omega([0, t]^2)\Big)^{\frac1{\rho''}}   \sum_{k=2}^4\Big(\omega_k([0,t]^2)\Big)^{\frac1{p'}}, 
  \end{align}
 where $\omega$ is the control given in \eqref{e:w}. This is our desired bound for the difference $\mathbf Dy^\pi- \mathbf Dy$.

Let us summarize our considerations so far. Gathering inequalities \eqref{e:y-y} and \eqref{e:Dy-Dy}, we have proved that 
\begin{align}
&\|y^\pi-y\|^2_{\ch^d}+\|\mathbf Dy^{\pi}- \mathbf Dy\|^2_{(\ch)^{\otimes 2}} \label{e:y-y-dy-dy}\\
\le& C\left(1+\|R\|_{\rho'-{\rm var};[0,T]^2}\right) \Big(\omega([0,t]^2)\Big)^{\frac1{\rho''}}\sum_{k=1}^4\Big(\omega_k([0,t]^2)\Big)^{\frac1{p'}} \sup_{i,j}\Big(\omega([t_i, t_{i+1}]\times[t_j, t_{j+1}])\Big)^{\frac1{\rho'}-\frac1{\rho''}}, \notag
\end{align}
where the controls $\omega, \omega_1, \omega_2, \omega_3, \omega_4$ are respectively defined by \eqref{e:w}, \eqref{e:w1}, \eqref{e:w2}, \eqref{e:w3} and \eqref{e:w4}. We can now argue as follows: first, 
 according to \eqref{e:bound-R-2rho}, \eqref{e:norms-compare} and \eqref{e:w} we have
\[\lim_{n\to \infty} \sup_{i,j} \omega([t_i, t_{i+1}]\times[t_j, t_{j+1}])=0. \]
Next we have assumed in Theorem \ref{thm} that 
\[\E\left[\|y\|^2_{p-{\rm var};[0,T]}+\|\mathbf D_0 y\|^2_{p-{\rm var};[0,T]}+\|\mathbf D y\|^2_{p-{\rm var};[0,T]^2}\right]<\infty. \]
Therefore one can take expected valued in \eqref{e:y-y-dy-dy} in order to get 
\[\lim_{n\to \infty} \E\left[\|y^{\pi}-y\|_{\ch^d}^2+\|\mathbf Dy^\pi- \mathbf Dy\|^2_{(\ch^d)^{\otimes 2}} \right]=0,\]
which is our claim. This concludes our proof. 
\end{proof}

 We close this section by showing that our main Theorem \ref{thm} generalizes previous Skorohod-Stratonovich integral correction formulae.

\begin{remark}[A comparison with \cite{HJT}]

In \cite{HJT}, the relationship \eqref{eq:strato-sko-1} is obtained for a $\gamma$-H\"older Gaussian process $x$ with $\gamma\in(0,1)$.  The process $y$ considered in \cite{HJT} is of the special form $y=f(x)$ for $f\in C^{2N}$ with $N=\lfloor\frac1\gamma\rfloor$.

Our main result Theorem \ref{thm} holds for Gaussian processes possessing finite $p$-variation with  $p\in(2, 3)$ (or $
\rho\in[1, \frac32)$). Noting that Propositions \ref{prop:bound-F} and \ref{prop:weighted-sum} hold under Hypotheses \ref{hyp:2d-var-R} and ~\ref{hyp:var-R}  for $p\in(2, 4)$, we believe that our approach could be extended to $p\in(2,4)$. A key difference between $p\in(2,3)$ and $p\in[3,4)$  is that a weighted sum in the third chaos of $x$ will be involved in the rough integral \eqref{b2} for $p\in[3,4)$. Thus for $p\in(2,4)$, to calculate the Skorohod-Stratonovich correction term, we also need develop some estimation for the weighted sum in third chaos of $x$ which is parallel to Proposition  \ref{prop:weighted-sum}.    

Note that the condition $p \in (2,3)$ ($\rho\in[1,\frac32)$) is also used to define the Young integrals appearing in \eqref{e:63}, \eqref{e:65} and \eqref{b5}. However, if we further assume that $R_\cdot$ and $R(\cdot, t)$ for each $t\in[0, T]$  are absolutely continuous  as in  \cite[Hypothesis 3.1]{HJT}, which is satisfied by fractional Brownian motion $B^H$ with Hurst parameter $H\in(0,1)$,   then  the integrals in \eqref{e:63}, \eqref{e:65} and \eqref{b5} are automatically well-defined as Riemann integrals.

\end{remark}

\begin{remark}[A comparison with \cite{CL}]
In \cite{CL}, the relationship between $\int_{0}^{t} y_{r} \dd \bx_{r}$
and $\int_{0}^{t} y_{r} \ddi x_{r}$ is studied, where $x$ satisfies Hypotheses \ref{hyp:2d-var-R} and \ref{hyp:var-R},   and $y$ is the solution to \eqref{eq:rde} with $\sigma$ being sufficiently regular. Indeed, under the conditions assumed in the main Theorem in \cite{CL}, our main result Theorem \ref{thm} also holds. More specifically, $\E[\|y\|^2_{p-{\rm var};[0,T]}]<\infty$ is a consequence of \cite[Theorem 2.25]{CL}; 
 $\E[\|\mathbf{D}_0y\|^2_{p-{\rm var};[0,T]}]+\E[\|\mathbf{D}y\|^2_{p-{\rm var};[0,T]^2}]<\infty$ follows from $\mathbf D_sy_t=\1_{[0,t)}(s) J_t^{\bf X} (J_x^{\bf X})^{-1} \sigma(Y_s)$ and Theorem \cite[Theorem 2.27]{CL} (see also the end of the proof \cite[Proposition 4.10]{CL}). With those relations in mind, our main Theorem \ref{thm} also covers the analysis performed in \cite{CL}.
\end{remark}

\medskip

\paragraph{\textbf{Acknowledgment}}
We would like to thank Tom Cass for some interesting discussions about this project.


\medskip

\begin{thebibliography}{99}


\bibitem{BC} F. Baudoin and L. Coutin.  Operators associated with a stochastic differential equation driven by fractional Brownian motions. {\em Stochastic Processes and their Applications.} 117 (2007), 550-574.


\bibitem{CFV}
T. Cass, P. Friz, N. Victoir: 
Non-degeneracy of Wiener functionals arising from rough differential equations. 
{\it Trans. Amer. Math. Soc.} {\bf 361} (2009), no. 6, 3359--3371. 

\bibitem{CHLT} T. Cass, M. Hairer, C. Litterer, S. Tindel:
Smoothness of the density for solutions to Gaussian Rough Differential Equations.
{\it Ann. Probab.}, {\bf 43} (2015), no. 1, 188-239.

\bibitem{CL}
T. Cass, N. Lim:
A Stratonovich-Skorohod integral formula for Gaussian rough paths.
\textit{The Annals of Probability}
2019, Vol. 47, No. 1, 1-60.

\bibitem{DT09} A. Deya and S. Tindel. Rough Volterra equations1: the algebraic integration setting. {\em Stochastics and Dynamics.} Vol. 9, No. 3 (2009) 437-477. 

\bibitem{DT11} A. Deya and S. Tindel. Rough Volterra equation 2: convolutional generalized integrals. {\em Stochastic Processes and their Applications.} 121 (2011), 1864-1899.


\bibitem{FGGR13}
P.K. Friz, B. Gess, A. Gulisashvili, S. Riedel: Jain-Monrad criterion for rough paths and applications to random Fourier series and non-Markovian H\"ormander theory.
{\it Ann. Probab.}, {\it 44} (1),  684-738 (2016).

\bibitem{FH} 
Friz, P. K. and Hairer, M., 
\textit{A course on rough paths: with an introduction to regularity structures,} Universitext, Springer (2014). 


\bibitem{fv11} P. Friz and N. Victoir.  A note on higher dimensional p-variation. {\em Electronic Journal of Probability} 2011, Vol. 16, No. 68, 1880-1899. 
 
\bibitem{fv10} P. Friz and N. Victoir.  Differential equations driven by Gaussian signals. {\em Annales de l'Institut Henri Poincar\'e - Probabilit\'es et Statistiques} 2010, Vol. 46, No. 2, 369-413. 


\bibitem{FV-bk}
P. Friz, N. Victoir:
\emph{Multidimensional dimensional processes as rough paths.}
Cambridge University Press (2010).

\bibitem{GOT}
B. Gess, C. Ouyang, S. Tindel:
Density bounds for solutions  to differential equations driven by Gaussian rough paths.
{\it J. Theoret. Probab.} {\bf 33} (2020), no. 2, 611-648. 

\bibitem{Gu}
M. Gubinelli:
Controlling rough paths.
{\it J. Funct. Anal.} {\bf 216}, 86-140 (2004).

\bibitem{HT}
F. Harang, S. Tindel:
Volterra equations driven by rough signals.
{\it Arxiv} preprint (2019).

\bibitem{HJT}
Y. Hu, M. Jolis, S. Tindel:
On Stratonovich and Skorohod stochastic calculus for Gaussian processes.
{\it Ann. Probab.}  {\bf 41}  (2013),  no. 3, 1656--1693.

\bibitem{HY} Y. Hu, J. Yan. Wick Calculus for Nonlinear Gaussian Functionals. {\it Acta Mathematicae Applicatae Sinica}, English Series Vol. 25, No. 3 (2009) 399-414.

\bibitem{LT}
Y. Liu, S. Tindel:
First-order Euler scheme for SDEs driven by fractional Brownian motions: the rough case.
{\it Ann. Appl. Probab.} {\bf 29} (2019), no. 2, 758-826.


\bibitem{NNT} A. Neuenkirch, I. Nourdin and S. Tindel. Delay equations driven by rough paths. \emph{ Electon. J. Probab.} {\bf 13} (2008), no. 67, 2031-2068.

\bibitem{NN}
Nourdin, Ivan; Nualart, David Central limit theorems for multiple Skorokhod integrals. 
{\em J. Theoret. Probab. 23} (2010), no. 1, 39-64.

\bibitem{Nu06}
D. Nualart: \emph{The Malliavin Calculus and Related Topics.} Probability and its Applications. Springer-Verlag, 2nd
Edition, (2006).

\bibitem{NS}
D. Nualart, B. Saussereau:
Malliavin calculus for stochastic differential equations driven by a fractional  Brownian motion.  {\it Stochastic Process. Appl.}  {\bf 119}  (2009),  no. 2, 391--409. 

\bibitem{NTa}
D. Nualart, M. Taqqu:
Wick-It\^o formula for Gaussian processes.
{\it Stoch. Anal. Appl.} {\bf 24} (2006), 599--614.

\bibitem{Tow02}
Nasser Towghi, 
Multidimensional extension of L. C. Young's inequality, 
\emph{J. Inequal. Pure Appl. Math.} {\bf 3} (2002), no. 2, Article 22, 13 pp. (electronic).

\bibitem{SC15} S. Tindel and K. Chouk. 
Skorohod and Stratonovich integration in the plane. \emph{ Electon. J. Probab.} {\bf 20} (2015), no. 39, 1-39. 
\end{thebibliography}
\end{document}